\documentclass[10pt,twocolumn,twoside]{IEEEtran}

\IEEEoverridecommandlockouts                             
\usepackage[pdftex]{graphicx}
\graphicspath{{figure/}}
\usepackage{amsmath}
\usepackage{amssymb}
\usepackage{url}

\usepackage{amsthm} 
\usepackage{cite}
\usepackage{color}
\usepackage{epstopdf}

\usepackage{bm}
\usepackage{comment}

\usepackage{algorithm}
\usepackage{algorithmic}

\newcommand{\bP}{\bm{P}}
\newcommand{\bnu}{\bm{\nu}}
\newcommand{\bpi}{\bm{\pi}}
\newcommand{\be}{\bm{e}}
\newcommand{\bd}{\bm{d}}

\newcommand{\bw}{\bm{\omega}}

\newcommand{\bpsi}{\bm{\psi}}
\newcommand{\bth}{\bm{\theta}}
\newcommand{\bla}{\bm{\lambda}}
\newcommand{\bsi}{\bm{\sigma}}
\newcommand{\bga}{\bm{\gamma}}
\newcommand{\bmu}{\bm{\mu}}
\newcommand{\br}{\bm{r}}
\newcommand{\bx}{\bm{x}}
\newcommand{\by}{\bm{y}}
\newcommand{\bz}{\bm{z}}
\newcommand{\bo}{\bm{0}}
\newcommand{\Dz}{\Delta \bm{z}}
\newcommand{\Dd}{\Delta \bm{d}}
\newcommand{\DP}{\Delta \bm{P}}
\newcommand{\Dpsi}{\Delta \bm{\psi}}
\newcommand{\Dmu}{\Delta \bm{\mu}}
\newcommand{\Dw}{\Delta \bm{\omega}_\mathcal{G}}

\newcommand{\bze}{\bm{\zeta}}

\newtheorem{theorem}{\textbf{Theorem}}
\newtheorem{proposition}{\textbf{Proposition}}

\newtheorem{lemma}{\textbf{Lemma}}
\newtheorem{assumption}{\textbf{Assumption}}
\newtheorem{remark}{\textbf{Remark}}

\newcommand{\add}[1]{{\color{black}{#1}}}
\allowdisplaybreaks

\allowdisplaybreaks

\title{Distributed Automatic Load-Frequency Control with Optimality in Power Systems}

\author{Xin Chen, 
	Changhong Zhao,
	~Na Li
	\thanks{  X. Chen and N. Li are with the School of Engineering and Applied Sciences, Harvard University, USA. Email: (chen\_xin@g.harvard.edu, nali@seas.harvard.edu).}
	\thanks{C. Zhao is with the Department of Information Engineering, the Chinese University of Hong Kong, China. Email:  zhchangh1987@gmail.com.}
	\thanks{ 
		The work was supported by
		NSF 1608509, NSF CAREER 1553407, AFOSR YIP, and ARPA-E through the NODES program.
} 
}

\begin{document}

\maketitle

\begin{abstract}
	With the increasing penetration of renewable energy resources, power systems face new challenges in  maintaining power balance and the nominal frequency. This paper studies load control to handle these challenges. In particular, a fully distributed automatic load control (ALC) algorithm, which only needs local measurement and local communication, is proposed. We prove that the load control algorithm globally converges to an optimal operating point which minimizes the total disutility of users, restores the nominal frequency and the scheduled tie-line power flows, and respects the load capacity limits and the thermal constraints of transmission lines.
	 It is further shown that the asymptotic convergence still holds even when inaccurate system parameters are used in the control algorithm.  
	\add{In addition, the global exponential convergence of the reduced ALC algorithm without considering the capacity limits is proved and leveraged to study the dynamical tracking performance and robustness of the algorithm.	}
	Lastly, the effectiveness, optimality, and robustness of the proposed algorithm are demonstrated via numerical simulations. 
\end{abstract}


\section{Introduction}

\IEEEPARstart{I}{n} power systems, generation and load are required to be balanced all the time. Once a mismatch between generation and load occurs, the system frequency will deviate from the nominal value, e.g., 50 Hz or 60 Hz, which may undermine the electric facilities and even cause system collapse.  Hence, it is crucial to maintain the frequency closely around its nominal value. Traditionally, the generator-side control \cite{3m} plays a dominant role in frequency regulation, where the generation is managed to follow the time-varying load. However, \add{with the rapid proliferation of renewable energy resources, such as wind power and solar energy, it becomes more challenging to maintain power balance and the nominal frequency due to the increasing volatility in renewable generation.
}

To address these challenges, as a promising complement to generation control, load control has received considerable attention in the recent decade.
Because controllable loads are ubiquitously distributed in power systems and can respond fast to regulation signals or frequency deviation \cite{r_r}. 
There has been a large amount of research effort devoted to frequency regulation provided by controllable loads, including electric vehicles \cite{ev1, ev2}, heating, ventilation and air-conditioning systems \cite{hvac1}, energy storage systems \cite{bat1, bat2}, and thermostatically controlled loads \cite{tcl}. Several demonstration projects \cite{p1, p2, p3} verified the viability of load-side participation in frequency regulation. The literature above focuses on modeling and operating the loads for frequency regulation, and leaves the development of system-wide optimal load control techniques as an unresolved task.

For load-side frequency control, centralized methods \cite{cen1, cen2} need to exchange information over remotely connected control areas, which imposes a heavy communication burden with expanded computational and capacity complexities \cite{lr2}. This concern motivates a number of studies on distributed control methods. In \cite{dec1, dec2, dec3}, load control is implemented by solving a centralized optimization problem using appropriate decomposition methods. The decomposition methods generate optimal control schemes that respect the operational constraints, but their convergence relies on network parameters.
In \cite{pi1}, a distributed proportional-integral (PI) load controller is designed to attenuate constant disturbances and improve the dynamic performance of the system, whereas operational constraints, such as load power limits and line thermal constraints, are not taken into account. References \cite{re1, sm,the} reversely engineer power system dynamics as primal-dual algorithms to solve optimization problems for frequency regulation, and prove global asymptotic stability of the closed-loop system independently of control parameters. \add{Specifically, reference \cite{re1}
 studies the economic \emph{automatic generation control} (AGC) mechanism and develops a distributed generator control scheme for frequency regulation. In \cite{sm}, a distributed load control method is proposed for \emph{primary frequency regulation},  which can only stabilize the frequency but not restore the nominal value. Reference \cite{the} is the most related work, which inspires this paper, while the key differences between the load control algorithms in this paper and in \cite{the}  are elaborated as Remark \ref{diff}.
 }

In this paper, we develop a fully distributed automatic load control (ALC) method for secondary frequency regulation. It can eliminate power imbalance, restore nominal system frequency, and maintain scheduled tie-line power flows in a manner that minimizes the total disutility of load adjustment.  
The development of the proposed ALC method is based on the interpretation of the closed-loop system dynamics as a primal-dual algorithm to solve a well-designed optimal load control  problem. The main contributions of this paper are twofold:
\add{
\begin{itemize}
    \item [1)] The sensing requirement and communication requirement are greatly alleviated with the proposed ALC method. Precisely, 
    the information of instant power imbalance is completely circumvented in the control process, and only  local measurement and local communication are required, which warrants a fully distributed operation mode.
    The key for 
     achieving these properties is a new reformulation (ref. model (\ref{op})) of the optimal load control problem, whose 
      partial primal-dual gradient flow with the variable substitution technique leads to 
    the design of  the proposed ALC algorithm.
    
    \item [2)] In addition to establishing the global asymptotic convergence of the ALC algorithm, we further prove the global \emph{exponential} convergence of the reduced ALC algorithm without considering the capacity limits. Then this fast convergence property is  leveraged to provide theoretic guarantees on  the algorithm's dynamical tracking performance and robustness. 
    The crux to prove the global exponential convergence is the novel design of a quadratic
    Lyapunov function (\ref{lya}) with non-zero off-diagonal terms.  
\end{itemize}
}

These contributions overcome the main limitations in the existing approaches reviewed above and facilitate practical implementations of the proposed ALC algorithm. 
Lastly, the effectiveness,  optimality,  and  robustness  of  the  proposed ALC algorithm are  demonstrated  via  numerical  simulations on the 39-bus New England power system 
using Power System Toolbox (PST) \cite{pst1}.


The remainder of this paper is organized as follows: Section \ref{sec:model} introduces the power network dynamic model and formulates the optimal load control problem. Section \ref{sec:algorithm} presents the proposed ALC algorithm and its global asymptotic convergence. 
 Section \ref{sec:dynamic} analyzes the global exponential convergence of the reduced ALC algorithm and its dynamical tracking error.
  Numerical tests are carried out in Section \ref{sec:simulations}, and conclusions are drawn in Section \ref{sec:conclusion}.

\add{

\textbf{Notations.} 
 Boldface letters are used for column vectors.
 $|\cdot|$ takes entry-wise absolute value of a vector (scalar) or denotes the cardinality of a set. $||\cdot||$ denotes the 2-norm of a vector or the induced 2-norm for matrices, and $||\bx||_Q:=\sqrt{\bx^\top Q\bx}$ with $Q\succeq 0$. We use $(\cdot)^\top$ for matrix transposition and $(\cdot)^{-1}$ for matrix inverse. For any two vectors $\bx,\by$, 
$[\bx;\by]: = [\bx^\top, \by^\top]^\top $ denotes their column merge.

}

\section{System Model and Problem Formulation}\label{sec:model}

\subsection{Dynamic Network Model }

Consider a power network delineated by a graph $G({\mathcal{N}},\mathcal{E})$, where  ${\mathcal{N}}:=\left\{1,\cdots,|\mathcal{N}|\right\}$  denotes the set of buses and $\mathcal{E}\subset {\mathcal{N}} \times {\mathcal{N}}$  denotes the set of transmission lines connecting the buses. Suppose that $G({\mathcal{N}},\mathcal{E})$ is connected and directed with arbitrary directions assigned to the transmission lines. Note that if $ij\in \mathcal{E}$, then $ji\not\in\mathcal{E}$. The buses $i\in \mathcal{N}$ are divided into two types: generator buses and load buses, which are denoted respectively by the sets $\mathcal{G}$ and $\mathcal{L}$ with $\mathcal{N}=\mathcal{G}\cup\mathcal{L}$. 
 A generator bus is connected to generators and may also have
loads attached, while a load bus is only connected to loads.

{\color{black} For notational simplicity, all the variables in this paper represent the deviations from their nominal values  that are determined by the previous solution of economic dispatch}. We consider the direct current (DC) power flow model \cite{dc, dc2}: 
\begin{eqnarray} \label{DC}
\qquad \qquad P_{ij} &=& B_{ij}\left(\theta_i-\theta_j\right)\quad\qquad \forall ij \in\mathcal{E}
\end{eqnarray}
where $P_{ij}$ is the active power flow on line $ij$, and $\theta_i$ denotes the voltage phase angle of bus $i$. $B_{ij}$ is a network constant defined by
$$B_{ij}:=\frac{|V_i||V_j|}{x_{ij}}\cos \left(\theta_i^0-\theta_j^0 \right)$$
where $|V_i|, |V_j|$ are the voltage magnitudes at buses $i$ and $j$ (which are assumed to be constant in the DC model) and $x_{ij}$ is the reactance of line $ij$ (which is assumed to be purely inductive in the DC model). $\theta_i^0$ is the nominal voltage phase angle of bus $i$. See \cite{sm} for a detailed description.

The dynamical model of the power network is
\begin{subequations} \label{dynamic}
	\begin{align}
	M_i \dot{\omega}_i &=- \left( D_i\omega_i+d_i-P_i^{in}+\sum_{j:ij\in \mathcal{E}}P_{ij} -\sum_{k:ki\in \mathcal{E}}P_{ki} \right)  \nonumber
	\\
	& \qquad \qquad\qquad\qquad\qquad\qquad\qquad\quad \forall  i\in \mathcal{G} \label{dynamic:g}\\
	0 & = D_i\omega_i+d_i-P_i^{in}+\sum_{j:ij\in \mathcal{E}}P_{ij}-\sum_{k:ki\in \mathcal{E}}P_{ki} \nonumber
	\\
	& \qquad \qquad\qquad\qquad\qquad\qquad\qquad\quad\forall i\in \mathcal{L} \label{dynamic:l}\\
	\dot{P}_{ij}& =B_{ij}\left(\omega_i-\omega_j\right) \qquad\qquad\qquad\qquad  \forall ij\in \mathcal{E} \label{dynamic:flow}
	\end{align}
\end{subequations}
where $\omega_i$ denotes the frequency, $M_i$ is the generator inertia constant, and $D_i$ is the damping coefficient, at bus $i$. The controllable load at bus $i$ is denoted by $d_i$, and the other uncontrollable power injection (the generation minus uncontrollable frequency-insensitive load) at bus $i$ is denoted by $P_i^{in}$. 

Equations \eqref{dynamic:g} and \eqref{dynamic:l} describe the frequency dynamics at generator buses and load buses, respectively. Actually, they both indicate power balance at every time instant of the dynamics, as illustrated in Figure \ref{dyna}. The damping term $D_i\omega_i=(D_i^g+D_i^l)\, \omega_i$ characterizes the total effect of generator friction and frequency-sensitive loads. The line flow dynamics is delineated by \eqref{dynamic:flow}. The model \eqref{dynamic} essentially assumes that the frequency deviation  is small at every bus. See \cite{sm} for a justification of the model \eqref{dynamic}. 

\begin{figure}[thpb]
	\centering
	\includegraphics[scale=0.62]{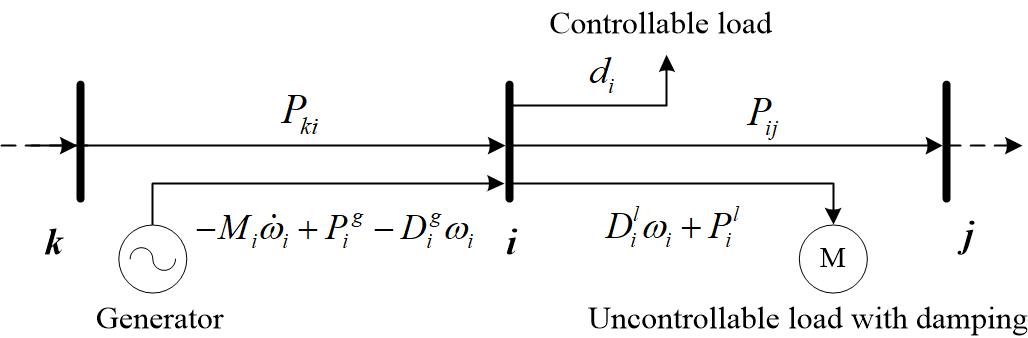}
	\caption{Frequency dynamics at bus $i$, where $P_i^g$ and $P_i^l$ denote generator mechanical power and uncontrollable frequency-insensitive load, respectively; $D^g_i$ and $D^l_i$ denote the damping coefficients of generators and loads, respectively.}
	\label{dyna}
\end{figure}

\begin{remark}The simplified linear model \eqref{dynamic} is employed for the purpose of algorithm design and stability analysis. The ALC algorithm that will be developed later can be applied to power systems with more complex and nonlinear dynamics. In Section \ref{sec:simulations}, a high-fidelity power system simulator  is used to test the ALC algorithm on a realistic dynamical model.
\end{remark}

\subsection{Optimal Load Control Problem}\label{subsec:OLC}

 Given a step change of uncontrollable power injection, i.e. $\bP^{in}:=\left(P_i^{in}\right)_{i\in\mathcal{N}}$, we adjust  controllable loads $\bd:=\left(d_i\right)_{i\in\mathcal{N}}$ for frequency regulation and the control goals are listed as follows:
\begin{enumerate}
\item Restore the system frequency to its nominal value.
\item Rebalance the system power while making each control area absorb its own power change, so that the scheduled tie-line power transfers are restored.
\item Modulate the controllable loads in an economically efficient way that minimizes the total disutility of load adjustment, while satisfying critical operational constraints including load power limits and line thermal limits.
\end{enumerate}

The second and third control goals can be formulated as the following optimal load control (OLC) problem:
\begin{subequations} \label{olc}
\begin{align}
\mathrm{Obj.} \ \, &\min_{\bd,\bth} \quad \sum_{i\in \mathcal{N}}c_i\left(d_i\right) \label{olc:obj}\\
\begin{split}
    \mathrm{s.t.} \ \, &d_i=P^{in}_i-\sum_{j:ij\in \mathcal{E}_{in}}B_{ij}\left(\theta_i-\theta_j\right)\\
    &\qquad\qquad +\sum_{k:ki\in \mathcal{E}_{in}}B_{ki}\left(\theta_k-\theta_j\right) \quad \forall i\in \mathcal{N}
\end{split} \label{olc:balance}\\
	&\underline{d}_i \leq d_i \leq \overline{d}_i  \qquad \qquad \qquad  \qquad \qquad \forall i\in \mathcal{N} \label{olc:power_limits}\\
	& \underline{P}_{ij} \leq B_{ij}\left(\theta_i-\theta_j\right) \leq  \overline{P}_{ij}\qquad \ \ \,\  \quad   \forall ij\in\mathcal{E} \label{olc:thermal_limits}
	\end{align}
\end{subequations}
where $\mathcal{E}_{in}$ denotes the subset of lines that connect buses within the same control area. Constants $\overline{d}_i$ and $\underline{d}_i$ are the upper and lower load power limits at bus $i$, respectively; and $\overline{P}_{ij}$ and $\underline{P}_{ij}$ specify the thermal limits of line $ij$. The function $c_i(d_i)$ quantifies the cost or disutility for load adjustment.

The objective \eqref{olc:obj} is to minimize the total cost of load adjustment. Equation \eqref{olc:balance} guarantees that the power imbalance is eliminated \emph{within} each control area; this can be shown by summing \eqref{olc:balance} over the buses in the same area $\mathcal{A}$, which leads to $\sum_{i\in\mathcal{A}}d_i=\sum_{i\in\mathcal{A}}P^{in}_i$. Equations \eqref{olc:power_limits} and \eqref{olc:thermal_limits} impose the load power constraints and the line thermal constraints, respectively. A load control scheme is considered to be optimal if it leads to a steady-state operating point which is a solution to the OLC problem \eqref{olc}.

 To facilitate the subsequent proof of convergence, we make the following assumptions:

\begin{assumption}\label{assumption:convexity}
	For $i \in \mathcal{N}$, the cost function $c_i(\cdot)$ is strictly convex and continuously differentiable.
\end{assumption}

\begin{assumption}\label{assumption:feasibility}
      The OLC problem \eqref{olc} is feasible.
\end{assumption}

\section{Optimal Automatic Load Control}\label{sec:algorithm}

In this section, a fully distributed ALC scheme (see Algorithm 1) is developed for secondary frequency regulation. 
The basic approach of controller design is \emph{reverse and forward engineering} \cite{the,re1,sm}, which interprets the system dynamics as a primal-dual gradient algorithm to solve a reformulated OLC problem.

\subsection{Reformulated Optimal Load Control Problem}

To explicitly take into account the first control goal in Section \ref{subsec:OLC}, i.e., restoring nominal frequency, the OLC problem \eqref{olc} is reformulated as follows:
\begin{subequations} \label{op}
	\begin{align}
	\mathrm{Obj.} \ \ &\min_{\bd,\bw, \bP, \bpsi} \quad \sum_{i\in \mathcal{N}}c_i\left(d_i\right)+\sum_{i\in \mathcal{N}}\frac{1}{2} D_i \omega_i^2\\
	\begin{split}
    	\mathrm{s.t.} \ \  &d_i=P^{in}_i-D_i\omega_i-\sum_{j:ij\in \mathcal{E}}P_{ij}+\sum_{k:ki\in \mathcal{E}}P_{ki}
    	\\
    	&\qquad\qquad\qquad\qquad\qquad\qquad\qquad\quad\ \, \forall i \in \mathcal{N} \end{split} \label{op:balance}\\
	& \underline{d}_i \leq d_i \leq \overline{d}_i  \qquad\qquad\qquad\qquad\qquad~  \forall i \in \mathcal{N} \label{op:power_limits}\\
	\begin{split}
	& d_i=P^{in}_i-\sum_{j:ij\in \mathcal{E}_{in}}B_{ij}\left(\psi_{i}-\psi_j\right) \\
	&\qquad  \quad \ \ +\sum_{k:ki\in \mathcal{E}_{in}}B_{ki}\left(\psi_k-\psi_i\right) \  \quad\forall i \in \mathcal{N}
	\end{split}\label{op:virtual_balance}\\
	&\underline{P}_{ij} \leq B_{ij}\left(\psi_{i}-\psi_j\right) \leq \overline{P}_{ij}  \qquad\qquad \, \forall ij \in \mathcal{E} \label{op:thermal_limits}
	\end{align}
\end{subequations}
where $\psi_{i}$ is an auxiliary variable interpreted as the virtual phase angle of bus $i$, and  $B_{ij}\left(\psi_{i}-\psi_j\right)$ is the virtual power flow on line $ij$.  Define vectors $\bw:=\left(\omega_i\right)_{i\in\mathcal{N}}$, $\bd:=\left(d_i\right)_{i\in\mathcal{N}}$, $\bP:=\left(P_{ij}\right)_{ij\in\mathcal{E}}$, and $\bpsi:=\left(\psi_i\right)_{i\in\mathcal{N}}$.

In the reformulated OLC problem \eqref{op}, the virtual phase angles $\bpsi$ is introduced to constrain the real power flow. 
See \cite{the} for detailed explanations, where the concepts of virtual phase angle and virtual power flow are first proposed.
Constraints \eqref{op:balance} and \eqref{op:virtual_balance} are introduced so that the primal-dual gradient algorithm solving \eqref{op} is exactly the power network dynamics under proper control.
The equivalence between problems \eqref{olc} and \eqref{op} is established as follows.

\begin{lemma}
	Let $\left(\bw^*,\bd^*,\bP^*,\bpsi^* \right)$ be an optimal solution of problem \eqref{op}. Then $\omega^*_i=0$ for all $i\in\mathcal{N}$, and $\bd^*$ is optimal for problem (\ref{olc}).
\end{lemma}

\begin{proof}
Let $\left(\bw^*,\bd^*,\bP^*,\bpsi^* \right)$ be an optimal solution of \eqref{op}, and assume that $\omega^*_i\not=0$ for some $i\in\mathcal{N}$. The optimal objective value of \eqref{op} is therefore:
\begin{eqnarray}\nonumber
f^*&=&\sum_{i\in \mathcal{N}}c_i\left(d_i^*\right)+\sum_{i\in \mathcal{N}}\frac{1}{2}D_i \left(\omega_i^*\right)^2.
\end{eqnarray}
Then consider another solution $\left\{\bw^o,\bd^*,\bP^o,\bpsi^* \right\}$ with $\omega^o_i=0$ for $i\in\mathcal{N}$, $	P_{ij}^o=B_{ij}\left(\psi_i^*-\psi_j^*\right)$ for $ij\in\mathcal{E}_{in}$, and $P_{ij}^o=0$ for $ij\in\mathcal{E}\backslash\mathcal{E}_{in}$. It can be  checked that this solution is feasible for problem (\ref{op}), and its corresponding objective value is
$$f^o=\sum_{i\in \mathcal{N}}c_i\left(d_i^*\right)< f^*$$
which contradicts the optimality of $\left(\bw^*,\bd^*,\bP^*,\bpsi^* \right)$. 
Hence $\omega^*_i=0$ for all $i\in\mathcal{N}$. 

Since constraints \eqref{olc:balance} and \eqref{op:virtual_balance} take the same form, when $\omega_i=0$ and given $(\bd, \bpsi)$, one can always find $\bP$ that satisfies \eqref{op:balance} by taking $	P_{ij}=B_{ij}\left(\psi_i-\psi_j\right)$ for $ij\in\mathcal{E}_{in}$ and $P_{ij}=0$ for $ij\in\mathcal{E}\backslash\mathcal{E}_{in}$.
Therefore the feasible set of \eqref{op} restricted to $\omega_i=0$ and projected onto the $(\bd, \bpsi)$-space is the same as the feasible set of \eqref{olc} on the $(\bd, \bth)$-space. As a result, for any $\left(\bw^*,\bd^*,\bP^*,\bpsi^* \right)$ that is an optimal solution of \eqref{op}, $\bd^*$ is also optimal for \eqref{olc}.
\end{proof}

\subsection{Automatic Load Control Algorithm} \label{sec_olc}

We design a partial primal-dual gradient method to solve the reformulated OLC problem \eqref{op}, so that the solution dynamics can be exactly interpreted as the power network dynamics with load frequency control. Based on this interpretation, the optimal ALC algorithm is developed.

The Lagrangian function of problem \eqref{op} is
\begin{eqnarray}   \label{lagr}
    &&L= \sum_{i\in \mathcal{N}}c_i\left(d_i\right)+\sum_{i\in \mathcal{N}}\frac{1}{2}D_i \omega_i^2 \nonumber \\
    &&+\sum_{i\in \mathcal{N}}\lambda_i\left( -d_i\!+\!P^{in}_i\!-\!D_i\omega_i\!-\!\sum_{j:ij\in \mathcal{E}}P_{ij}\!+\!\sum_{k:ki\in \mathcal{E}}P_{ki} \right) \nonumber\\
    &&+\sum_{i\in \mathcal{N}}\mu_i \left(-d_i+P^{in}_i-\sum_{j:ij\in \mathcal{E}_{in}}B_{ij}\left(\psi_{i}-\psi_j\right)\right. \nonumber \\
    && \left. \qquad\qquad\qquad\qquad\qquad\quad+\sum_{k:ki\in \mathcal{E}_{in}}B_{ki}\left(\psi_k-\psi_i\right)\right) \nonumber \\ 
    &&+\sum_{ij\in \mathcal{E}_{in}} \sigma_{ij}^+\left(B_{ij}\left(\psi_{i}-\psi_j\right)-\overline{P}_{ij}\right) \nonumber \\
    &&+ \sum_{ij\in \mathcal{E}_{in}} \sigma_{ij}^-\left(-B_{ij}\left(\psi_{i}-\psi_j\right)+\underline{P}_{ij}\right) \nonumber \\
    &&+\sum_{i\in \mathcal{N}}\gamma_i^+\left(d_i-\overline{d}_i\right)+\sum_{i\in \mathcal{N}}\gamma_i^-\left(-d_i+\underline{d}_i\right) \label{L}
\end{eqnarray}
where $\lambda_i,\mu_i$ are the dual variables associated with the equality constraints \eqref{op:balance} and \eqref{op:virtual_balance}, and $\gamma_i^+, \gamma_i^-, \sigma_{ij}^+, \sigma_{ij}^- \geq 0$ are the dual variables associated with the inequality constraints \eqref{op:power_limits} and \eqref{op:thermal_limits}. 
Define  $\bw_{\mathcal{G}}:=\left(\omega_i\right)_{i\in\mathcal{G}}$, $\bw_{\mathcal{L}}:=\left(\omega_i\right)_{i\in\mathcal{L}}$, $\bmu:=\left(\mu_i\right)_{i\in\mathcal{N}}$,
$\bsi:=\left(\sigma_{ij}^+,\sigma_{ij}^-\right)_{ij\in\mathcal{E}_{in}}$, and $\bga:=\left(\gamma_{i}^+,\gamma_{i}^-\right)_{i\in\mathcal{N}}$.

Then the partial primal-dual gradient method is given by the following three steps:

\underline{\textit{Step 1)}}: Solve $\min_{\bw} L$ by taking $\frac{\partial L}{\partial \omega_i} =0$ for $i\in\mathcal{N}$, which results in 
\begin{gather} 
\qquad\qquad\qquad\qquad \omega_i = \! \lambda_i \qquad\qquad \qquad \forall i\in \mathcal{N}  \label{sol:w}
\end{gather}
and we obtain
$$\hat L(\bd,\bP,\bpsi, \bla, \bmu, \bsi, \bga):=\min_{\bw} L(\bw, \bd,\bP,\bpsi, \bla, \bmu, \bsi, \bga)$$
Equation (\ref{sol:w}) exhibits the equivalence between $\omega_i$ and $\lambda_i$, hence we substitute $\bw$ for $\bla$ in $\hat{L}$ and other equations for algorithm design.

\underline{\textit{Step 2)}}: Solve $\max_{\bw_{\mathcal{L}}} \hat{L}$ by taking $\frac{\partial {\hat L}}{\partial \omega_i} =0$ for $i\in \mathcal{L}$, which results in
\begin{gather}  
0=\!d_i\!-\!P^{in}_i\!+\!D_i \omega_i\!+\!\sum_{j:ij\in \mathcal{E}}P_{ij}\!-\!\sum_{k:ki\in \mathcal{E}}P_{ki} \ \,\quad \forall   i\in \mathcal{L}  \label{sol:lam}
\end{gather}
and we obtain
$$\overline L(\bd,\bP,\bpsi, \bw_\mathcal{G}, \bmu, \bsi, \bga):=\max_{\bw_\mathcal{L}} \hat L(\bd,\bP,\bpsi, \bw, \bmu, \bsi, \bga) $$

\underline{\textit{Step 3)}}: Apply the the standard primal-dual gradient algorithm on the remaining variables  to find the saddle point of $\overline{L}$, and the solution dynamics is formulated as follows:
\begin{subequations}\label{gra}
    \begin{align}
    	\begin{split}
	\dot{\omega}_i&=\epsilon_{\omega_i}\left( P^{in}_i-d_i-D_i\omega_i-\sum_{j:ij\in \mathcal{E}}P_{ij}  +\sum_{k:ki\in \mathcal{E}}P_{ki} \right)  
	\end{split} \label{gra:lambda}\\
    \dot{P}_{ij}&=\epsilon_{P_{ij}}\left(\omega_i-\omega_j\right)  \label{gra:P}\\
	\dot{d}_i&=\epsilon_{d_i}\left( -c_i^\prime\left(d_i\right)+ \omega_i+\mu_i-\gamma_i^++\gamma_i^-\right) \label{gra:d}\\
   \begin{split}
	\dot{\psi}_{i}&=\epsilon_{\psi_i}\left[
	\sum_{j:ij\in\mathcal{E}_{in}}\left(\mu_i-\mu_j- \sigma_{ij}^++\sigma_{ij}^-\right)B_{ij} \right.\\ &\phantom{=\;\;}\left. 
	\qquad +\sum_{k:ki\in\mathcal{E}_{in}}\left(\mu_i-\mu_k+ \sigma_{ki}^+-\sigma_{ki}^-\right)B_{ki}\right]
	\label{gra:psi} \end{split}\\
   \dot{\gamma}_i^+&=\epsilon_{{\gamma}_i^+}\left[d_i-\overline{d}_i\right]^+_{{\gamma}_i^+}  \label{gra:gamma_plus} \\
    \dot{\gamma}_i^-&=\epsilon_{{\gamma}_i^-}\left[-d_i+\underline{d}_i\right]^+_{{\gamma}_i^-} \label{gra:gamma_minus}\\
 	\begin{split}
	\dot{\mu}_i&=\epsilon_{\mu_i}\left(P^{in}_i -d_i-\sum_{j:ij\in\mathcal{E}_{in}}B_{ij}\left(\psi_{i}-\psi_j\right)\right.\\
	&\phantom{=\;\;}\left. \ \ \qquad\qquad\qquad+\sum_{k:ki\in \mathcal{E}_{in}}B_{ki}\left(\psi_k-\psi_i\right)\right)  
	\end{split} \label{gra:mu}\\
	\dot{\sigma}_{ij}^+&=\epsilon_{\sigma_{ij}^+}\left [B_{ij}\left(\psi_{i}-\psi_j\right)-\overline{P}_{ij}\right]^+_{\sigma_{ij}^+} \label{gra:sigma_plus}\\
	\dot{\sigma}_{ij}^-&=\epsilon_{\sigma_{ij}^-}\left [-B_{ij}\left(\psi_{i}-\psi_j\right)+\underline{P}_{ij}\right]^+_{\sigma_{ij}^-} \label{gra:sigma_minus}
	\end{align}
\end{subequations}
where \eqref{gra:lambda} is for $i\in\mathcal{G}$, \eqref{gra:P} is for $ij\in\mathcal{E}$, \eqref{gra:d}--\eqref{gra:mu} are for $i\in\mathcal{N}$,  and \eqref{gra:sigma_plus}--\eqref{gra:sigma_minus} are for $ij\in\mathcal{E}_{in}$. The notations containing $\epsilon$ represent appropriately selected positive constant step sizes. The operator $[x]^+_y$ means positive projection \cite{pp}, which equals $x$ if either $x>0$ or $y>0$,  and 0 otherwise; thus it ensures $\sigma_{ij}^+, \sigma_{ij}^-, \gamma_i^+, \gamma_i^- \geq 0$. 

\add{Since the instant value of  $P^{in}_i$ is usually unknown and hard to procure in practice}, a new variable $r_i$ defined as follows is introduced to substitute $\mu_i$:
\begin{equation} \label{varsub}
r_i=\left\{
\begin{aligned}
&\frac{K_i}{\epsilon_{\mu_i}}\mu_i-\frac{K_i}{\epsilon_{\omega_i}}\omega_i  &\qquad \forall i\in \mathcal{G}  \\
&\frac{K_i}{\epsilon_{\mu_i}}\mu_i &\qquad \forall i\in \mathcal{L}
\end{aligned}
\right.
\end{equation}
where $K_i$ is a positive constant. In this way, the necessity to know $P^{in}_i$ is circumvented. Define $\br:=(r_i)_{i\in \mathcal{N}}$.

\add{Let $\epsilon_{\omega_i}=1/M_i$ and $\epsilon_{P_{ij}}=B_{ij}$, then equations (\ref{sol:lam}) (\ref{gra:lambda}) (\ref{gra:P}) are exactly the same as the network dynamics (\ref{dynamic}). Thus after the variable substitution, 
the solution dynamics (\ref{sol:lam})-(\ref{gra})  is equivalent to the ALC algorithm \eqref{cm} together with the network dynamics (\ref{dynamic}). 
This key property attributes to the deliberate design of the reformulated OLC problem (4) and the partial primal-dual gradient method. As a result, the local load controller only needs to execute the ALC algorithm (\ref{cm}), while the network dynamics (\ref{dynamic}) is the natural evolution of  the physical power system in response to the load adjustment. In this way, a portion of the solution dynamics, i.e., equations (\ref{sol:lam}) (\ref{gra:lambda}) (\ref{gra:P}), or (\ref{dynamic}), is outsourced to the power network physics, and the ALC algorithm just needs to take  measurement of the local frequency and power flow from the physical system. 
The whole design procedure for the distributed load controller is illustrated in Figure \ref{desro}.
}

\begin{figure}[thpb]
	\centering
	\includegraphics[scale=0.27]{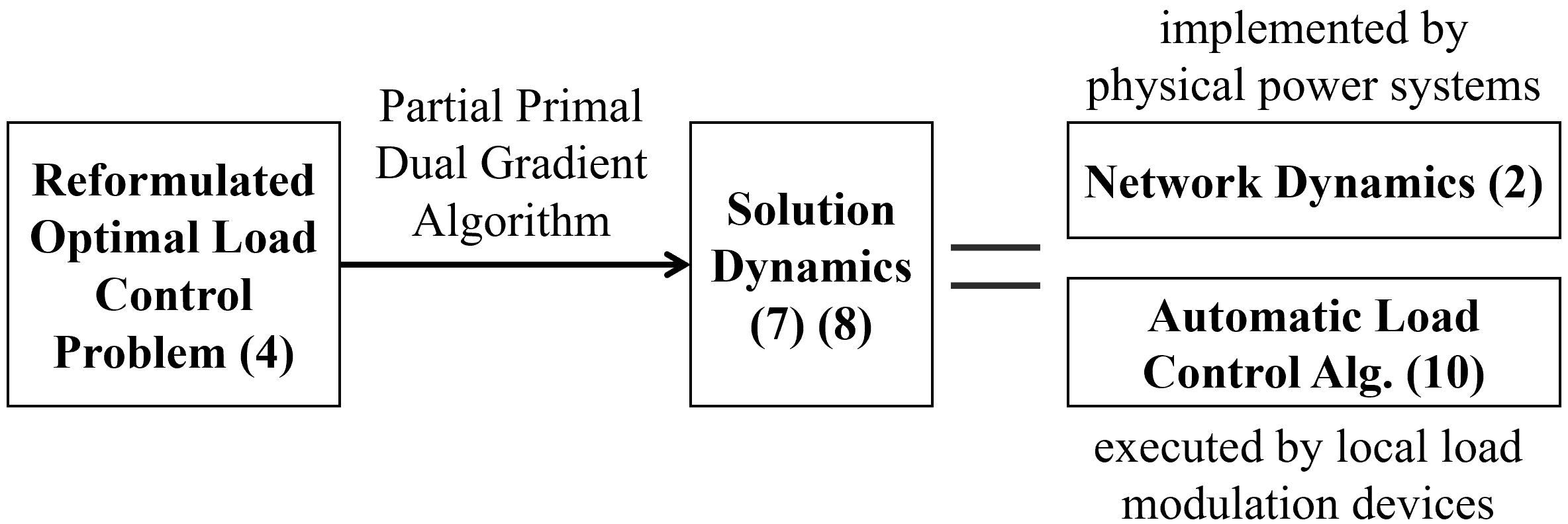}
	\caption{The design procedure for distributed automatic load controller.}
	\label{desro}
\end{figure}

\begin{algorithm}[htb] 
	\caption{Automatic Load Control Algorithm.} 
	\begin{algorithmic}
		\STATE	\begin{subequations} \label{cm}
			\begin{align} 
			\dot{d}_i&=\epsilon_{d_i}\left( -c_i^\prime\left(d_i\right)+\eta_i\omega_i+\frac{\epsilon_{\mu_i}}{K_i}r_i-\gamma_i^++\gamma_i^-\right) \label{cm:d}\\
			\begin{split}
			\dot{\psi}_{i}&=\epsilon_{\psi_i}\left[
			\sum_{j:ij\in\mathcal{E}_{in}}\left(\mu_i-\mu_j- \sigma_{ij}^++\sigma_{ij}^-\right)B_{ij} \right.\\
			&\phantom{=\;\;}\left. \quad\qquad+
			\sum_{k:ki\in\mathcal{E}_{in}}\left(\mu_i-\mu_k+ \sigma_{ki}^+-\sigma_{ki}^-\right)B_{ki}\right]
			\end{split} \label{cm:psi}\\
			\dot{\gamma}_i^+&=\epsilon_{{\gamma}_i^+}\left[d_i-\overline{d}_i\right]^+_{{\gamma}_i^+}\\
			\dot{\gamma}_i^-&=\epsilon_{{\gamma}_i^+}\left[-d_i+\underline{d}_i\right]^+_{{\gamma}_i^+}\\
			\begin{split}
			\dot{r}_i&=K_i\left[D_i\omega_i+\sum_{j:ij\in\mathcal{E}}P_{ij}-\sum_{k:ki\in\mathcal{E}}P_{ki}\right.\\
			&\phantom{=\;\;}\left.-\sum_{j:ij\in\mathcal{E}_{in}}B_{ij}\left(\psi_{i}-\psi_j\right)+\sum_{k:ki\in \mathcal{E}_{in}}B_{ki}\left(\psi_k-\psi_i\right)\right]
			\end{split} \label{cm:r}\\
			\dot{\sigma}_{ij}^+&=\epsilon_{\sigma_{ij}^+}\left [B_{ij}\left(\psi_{i}-\psi_j\right)-\overline{P}_{ij}\right]^+_{\sigma_{ij}^+}\\
			\dot{\sigma}_{ij}^-&=\epsilon_{\sigma_{ij}^-}\left [-B_{ij}\left(\psi_{i}-\psi_j\right)+\underline{P}_{ij}\right]^+_{\sigma_{ij}^-}
			\end{align}
		\end{subequations}	
	\end{algorithmic} 
\end{algorithm}

\begin{figure}[thpb]
	\centering
	\includegraphics[scale=0.27]{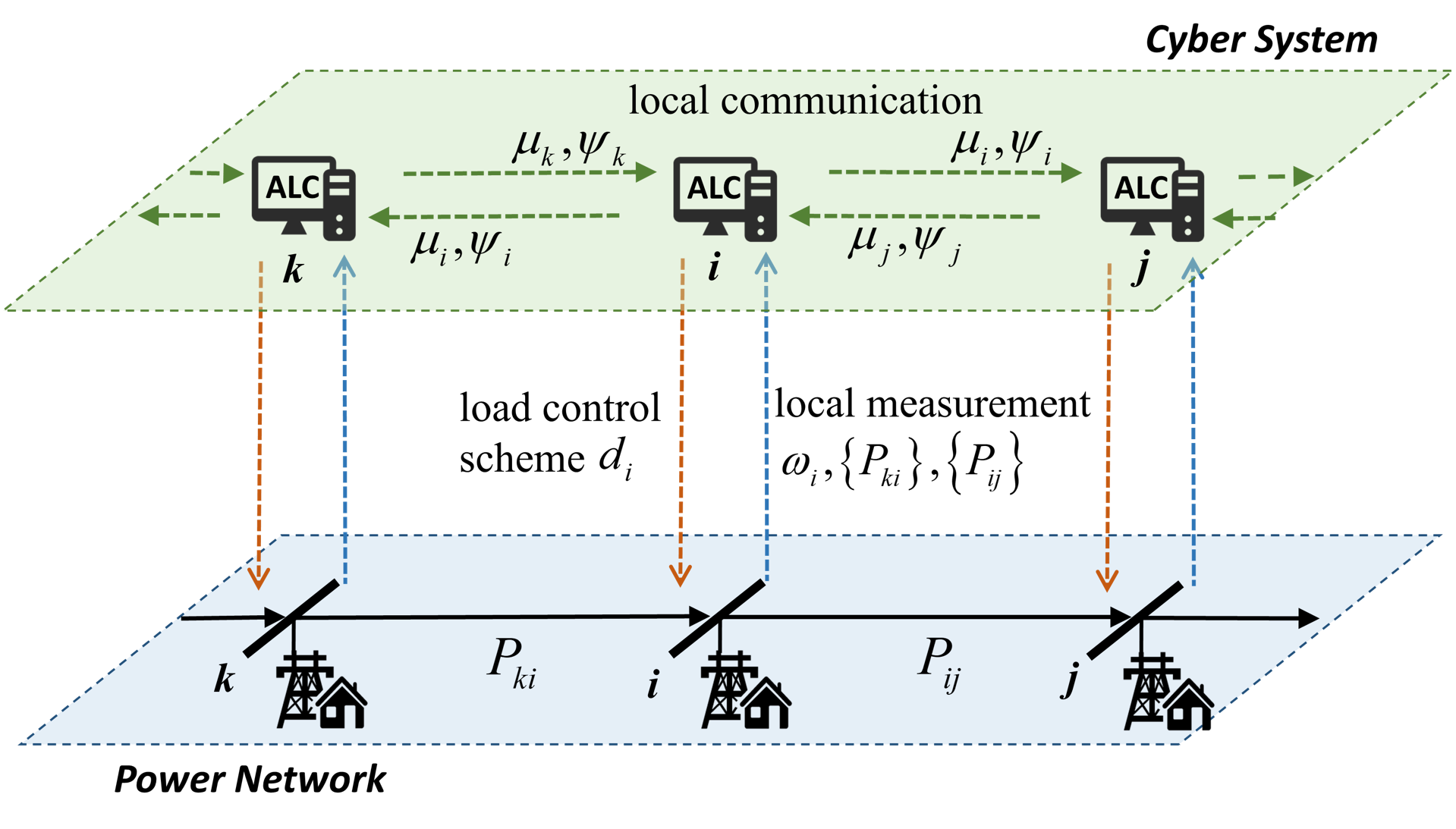}
	\caption{The automatic load control (ALC) mechanism.}
	\label{alg_l}
\end{figure}

In \eqref{cm:d}, $\eta_i$ is set as $({\epsilon_{\omega_i}+\epsilon_{\mu_i}})/{\epsilon_{\omega_i}}$ for $i\in \mathcal{G}$ and 1 for $i\in\mathcal{L}$ respectively. In \eqref{cm:psi}, $\mu_i$ is the abbreviation of the expression (\ref{mudef})
\begin{equation} \label{mudef}
\mu_i=\left\{
\begin{aligned}
&\frac{\epsilon_{\mu_i}}{\epsilon_{\omega_i}}\omega_i+\frac{\epsilon_{\mu_i}}{K_i} r_i  &\qquad \forall i\in \mathcal{G}  \\
&\frac{\epsilon_{\mu_i}}{K_i} r_i &\qquad \forall i\in \mathcal{L}
\end{aligned}
\right.
\end{equation}

The implementation of algorithm (\ref{cm}) is illustrated in Figure \ref{alg_l}. In the physical (lower) layer, each bus $i$ measures its own frequency deviation $\omega_i$ and the power flows $(P_{ki}, P_{ij})$ on its adjacent lines. In the cyber (upper) layer, each bus $i$ exchanges the information $(\mu_i,\psi_i)$ with its neighboring buses in the same control area. Then following algorithm \eqref{cm}, each bus $i$ updates the variables $(\psi_{i},\gamma_{i}, \sigma_{ij}, r_i)$ and computes its load adjustment $d_i$. Next, the control command $d_i$ is sent back to the physical layer and executed by the load modulation device. Afterwards, the system frequency and power flows respond to the load adjustment according to the physical law \eqref{dynamic}. In this manner, the combination of network dynamics (\ref{dynamic}) and the proposed control algorithm (\ref{cm}) forms a closed loop. Since only local measurement and local communication are required in this process,  the proposed ALC algorithm (\ref{cm}) is performed in a fully distributed manner.

\add{
\begin{remark}
Although the ALC algorithm (\ref{cm}) is developed based on step power changes, it is capable of handling continuous power disturbance. Because in practical implementation, the real-time measurements of frequency deviation and power flow are utilized to generate the load adjustment decisions, which renders the immediate  response to the time-varying power disturbance. The dynamical tracking performance of the ALC algorithm is  analyzed in Section \ref{sec-dyna}, and case studies on continuous power change are provided in Section \ref{ca-con}.
\end{remark}
}

\subsection{Asymptotic Convergence and Main Advantages} \label{sec-asymptotic}

In this part, we show that the proposed algorithm \eqref{cm} will converge to a steady-state operating point that is an optimal solution of the reformulated OLC problem \eqref{op}. This claim is restated formally as the following theorem.

\begin{theorem} \label{thm1}
Under  Assumption \ref{assumption:convexity} and \ref{assumption:feasibility}, the ALC algorithm \eqref{cm} together with the network dynamics \eqref{dynamic} globally asymptotically converges to a point $\left(\bd^*, \bw^*, \bP^*,  \bpsi^*, \bga^*, \br^*,\bsi^*\right)$, where $\left(\bd^*, \bw^*, \bP^*,  \bpsi^*\right)$ is an optimal solution of problem \eqref{op}.
\end{theorem}

\begin{proof} 
Since the closed-loop system dynamics \eqref{dynamic}, \eqref{cm} are equivalent to the solution dynamics  (\ref{sol:lam}), \eqref{gra}, we prove the convergence of dynamics (\ref{sol:lam}), \eqref{gra}  to an optimal solution of problem \eqref{op} instead. 

Define $\by:=[\bd; \bP; \bpsi; \bw_\mathcal{G};\bmu; \bsi;\bga]$ and let $\by^*$ be any equilibrium point of dynamics (\ref{gra}), which makes the right-hand-side of \eqref{gra} zero. 
Let $ \bw_\mathcal{L}^*$ be the solution of (\ref{sol:lam}) given $\by^*$. By Assumptions \ref{assumption:convexity} and \ref{assumption:feasibility}, strong duality holds for the problem \eqref{op}. Thus, according to \cite[Proposition 9]{re1}, $(\by^*,\bw_\mathcal{L}^*,\bla^*)$ with $\bla^*=\bw^*$ (\ref{sol:w}) is a saddle point of the Lagrangian $L$ (\ref{lagr}) and is primal-dual optimal for \eqref{op} \cite{boyd2004convex}.

Then we just need to prove that dynamics \eqref{gra} asymptotically converges to its equilibrium point $\by^*$. Since dynamics \eqref{gra} is obtained by applying the standard primal-dual gradient algorithm to solve the saddle point problem \eqref{nl}, i.e., 
\textit{Step 3)} in Section \ref{sec:algorithm}-B,
\begin{eqnarray} \label{nl}
\min_{\bd,\bP,\bpsi} \ \max_{\bw_\mathcal{G}, \bmu, \bsi\geq \bm{0}, \bga\geq \bm{0}} \overline{L}(\bd,\bP,\bpsi,\bw_\mathcal{G}, \bmu, \bsi, \bga)
\end{eqnarray}
the asymptotic convergence proof of dynamics \eqref{gra}  directly follows the results in \cite{pp,pf1}. Thus Theorem \ref{thm1} is proved.
\end{proof}

One challenge in implementing the ALC algorithm (\ref{cm}) is that 
 the damping coefficient $D_i$ is in general hard to know exactly. For this issue, we provide Theorem \ref{thm-robustD} in Appendix \ref{pf-robustD} to show 
  that the proposed load controller  is robust to the inaccuracy in $D_i$, in the sense that the ALC dynamics still converge to an optimal solution of the OLC problem, if the inaccuracy in $D_i$ is small and some additional conditions are satisfied.

\add{ 
\begin{remark} \label{diff}
Comparing with the load control scheme in reference \cite{the} (most related work to this paper), the key advantages of the proposed ALC algorithm (\ref{cm}) are 

1) (\textbf{Sensing Requirement}) To implement the load control scheme in \cite{the}, each bus  requires the value of the instant power change $P^{in}_i$  or the estimation of the angular acceleration $\dot{w}_i$, while their accurate values are hard to obtain in real-time application, especially for the aggregate bus with many generators and loads attached. In contrast, using a different design procedure, the proposed ALC algorithm (\ref{cm}) completely circumvents the information of $P^{in}_i$, and only the local measurements of $(\omega_i, P_{ki}, P_{ij})$ are required for each bus.

  2) (\textbf{Communication Requirement}) With the load control scheme in \cite{the}, each boundary bus needs to communicate with all the  other boundary buses within the same control area, 
  which may carry heavy remote communication burden, especially when two boundary buses are far away from each other; in addition, each boundary bus has to exchange information with its adjacent buses located in other control areas, which may violate the information privacy. 
  In contrast, using the ALC algorithm 
  (\ref{cm}), each bus 
  (no matter on boundary or not) only needs to communicate with its adjacent buses within the same control area, i.e., no information exchange among different control areas.

  Therefore, the sensing and communication requirements are greatly alleviated with the proposed ALC algorithm (\ref{cm}), which  renders a fully distributed control mechanism, while the 
  global asymptotical convergence can still be achieved.
\end{remark}
}

\add{
\subsection{Further Discussion}

In this paper, renewable generations are modelled as non-dispatchable power injection and captured by $\bP^{in}$. 
Actually, the proposed control algorithm that determines local load adjustment in real time can be applied to controlling the dispatchable renewable generation as well, without considering the inverter dynamics.
This setting is generally acceptable for practical application since the inverter dynamics  is much faster than the timescale of secondary frequency regulation. However, as the penetration of renewable generation deepens, 
the impacts of inverter dynamics and harmonics become more and more significant, therefore it is necessary  to model the internal dynamics of renewable sources in a realistic way. One of the future work is to design distributed inverter controller for 
 renewable energy sources  to provide frequency regulation and mitigate harmonics.

Besides, we make Assumption \ref{assumption:feasibility} to assume that each control area has sufficient controllable load/generation resources to absorb its own power change. Once a control area does not have enough controllable resources to eliminate the power imbalance, the OLC problem (\ref{olc}) becomes infeasible. In this situation, 
 the proposed load controller (\ref{cm}) can still work to exploit the limited resources to alleviate the frequency deviation, but the nominal frequency can not be restored. Hence, when the system operators suspect that a control area can not absorb the power change, they need to either 1) dispatch available load/generation resources from the neighbor control areas (i.e., relax the tie-line requirement), or 2) call upon more controllable resources, e.g., renewable generation or energy storage, for frequency regulation. For scheme 1), our proposed algorithm is easy to adjust to this situation by just modifying the set $\mathcal{E}_{in}$, then two or more control areas can be combined and share all the controllable resources. For scheme 2), as mentioned before, the proposed load control mechanism can be adapted to control the inverter-based renewable generations.

}

\section{Exponential Convergence, Dynamical Tracking and Robustness Analysis} \label{sec:dynamic}

This section studies the global exponential convergence of the ALC algorithm and analyzes its 
 dynamical tracking performance and robustness.

To facilitate theoretical analysis, we consider a system with sufficient capacities so that inequality constraints
(\ref{olc:power_limits}, \ref{olc:thermal_limits}) in the OLC problem (\ref{olc}) can be ignored, i.e., (\ref{op:power_limits}, \ref{op:thermal_limits}) in problem (\ref{op}). Then the reformulated OLC problem (\ref{op}) reduces to 
\begin{subequations} \label{op2}
	\begin{align}
	\mathrm{Obj.} \ \ &\min_{\bd,\bw, \bP, \bpsi} \quad c\left(\bd\right) + \frac{1}{2} \bw^\top D \bw \label{op2:obj}\\
	\mathrm{s.t.} \ \  &\bd=\bP^{in}-D\bw-A\bP \\
	& \bd=\bP^{in}- \bar{A}\bar B \bar A^\top \bpsi \label{op2:Y}
	\end{align}
\end{subequations}
where $c(\bd):=\sum_{i\in \mathcal{N}}c_i\left(d_i\right)$ and $D:=\text{diag}(D_i)_{i\in\mathcal{N}}$. $A$ is the node-branch incidence matrix with respect to the buses $i\in\mathcal{N}$ and the lines $ij\in\mathcal{E}$.  $\bar A$ is a sub-matrix of $A$, which is obtained by removing the columns associated with the boundary lines ($ij \in \mathcal{E} \backslash \mathcal{E}_{in}$) in $A$, and $\bar B := \text{diag} (B_{ij} )_{ij\in\mathcal{E}_{in}}$. 

Without loss of generality, we arrange the sequence of buses in vectors (matrices) so that  $\bP^{in} = [{\bP^{in}_\mathcal{G}}; {\bP^{in}_\mathcal{L}}]$ , $\bd = [\bd_\mathcal{G}; \bd_\mathcal{L}]$, $\bw = [\bw_\mathcal{G}; \bw_\mathcal{L}]$, $A = [A_\mathcal{G}; A_\mathcal{L}]$, and 
$D = \text{blockdiag}(D_\mathcal{G},D_\mathcal{L})$. 
Following the same solution procedure in Section \ref{sec:algorithm}-B, the ALC dynamics  (\ref{gra}) become
\begin{subequations}\label{gra2}
	\begin{align}
	\bm{0}&=\!\bd_\mathcal{L}\!-\!\bP^{in}_\mathcal{L}\!+\!D_\mathcal{L}\bw_\mathcal{L}\!+ A_\mathcal{L} \bP  \label{sol2:lam2} \\
		\dot{\bd}&=\Xi_{d}\cdot\left( -\nabla c(\bd)+ \bw+\bmu\right) \label{gra:d2}\\
			\dot{\bP}&= \Xi_{P} \cdot A^\top \bw           \label{gra:P2}\\
		\dot{\bpsi}&=\Xi_{\psi}\cdot     S \bmu          	\label{gra:psi2} \\
	\dot{\bw}_\mathcal{G}&= \Xi_{\omega}\cdot\left( -\bd_\mathcal{G}-D_\mathcal{G}\bw_\mathcal{G}-A_\mathcal{G} \bP + \bP^{in}_\mathcal{G}\right)  \label{gra:lambda2}\\
	\dot{\bmu}&=\Xi_{\mu}\cdot \left(  -\bd- S \bpsi    +\bP^{in}    \right)     \label{gra:mu2}
	\end{align}
\end{subequations}
where 
$S:= \bar{A}\bar B \bar A^\top$
and $\nabla c(\bd): = \left( c_i^\prime(d_i)  \right)_{i\in \mathcal{N}}$. Since the cost function $c(\bd)$ is a general convex function, it is noted that  (\ref{gra2}) is a \emph{nonlinear} dynamical system.



\subsection{Global Exponential Convergence Analysis}

The asymptotic convergence of the ALC algorithm has been 
exhibited in Theorem \ref{thm1}, while this part focuses on a stronger and highly desired property: global exponential convergence. To establish this, we firstly make Assumption \ref{ass:strong} for the cost function $c(\bd)$.




\begin{assumption} \label{ass:strong}
	For $i \in \mathcal{N}$, the cost function $c_i(\cdot)$ is  twice differentiable, $u$-strongly convex and $\ell$-smooth with $0<u\leq \ell$, i.e., $u\leq c_i''(d_i)\leq \ell$ for any $d_i$.
\end{assumption}

Let $\bz := \left[\bd; \bP;\bpsi;\bw_{\mathcal{G}}; \bmu \right]$ and $\bx:=\left[\bz; \bw_\mathcal{L}\right] $ be the system state.
Let  $\bx^* := [\bz^*;\bw_\mathcal{L}^*]$
be one of the equilibrium points of the ALC dynamics (\ref{gra2}). Define the equilibrium set $\mathcal{S}$ as (\ref{eq:S}):
    \begin{align} \label{eq:S}
    \begin{split}
        &\mathcal{S} : =  \left\{\bx\,|\,    {\bd} = \bd^*, {\bw}_\mathcal{G} = {\bw}_\mathcal{G}^*, {\bw}_\mathcal{L} = \bw_\mathcal{L}^*, \right.\\
	&\qquad\qquad\quad   \phantom{=\;\;}\left. {\bmu} =\bmu^*, A {\bP} = A\bP^*, S{\bpsi} = S \bpsi^*\right\}.
    \end{split}
    \end{align}
It can be checked that any point $\hat{\bx}\in \mathcal{S}$ is an equilibrium point of the ALC dynamics (\ref{gra2}), and thus the corresponding $(\hat{\bd}, \hat{\bw}, \hat{\bP},  \hat{\bpsi})$ is  an optimal solution of problem \eqref{op2} \cite{re1}.
Let $$\text{dist}(\bx,\mathcal{S}):=\inf_{\hat{\bx}\in\mathcal{S}}||\bx-\hat{\bx}||$$ denote the distance between   a point $\bx$ and the set  $\mathcal{S}$.
Then we have the following theorem:
\begin{theorem}\label{thm3}
    Under Assumption \ref{assumption:feasibility} and \ref{ass:strong}, the ALC dynamics (\ref{gra2})
    globally exponentially converge to the equilibrium set  $\mathcal{S}$ (\ref{eq:S}), 
    in the sense that there exist constants $C_0\geq 0$ and $\rho_0>0$ such that the distance between  $\bx(t)$ and $\mathcal{S}$ satisfies 
    \begin{align}
      \textnormal{dist}(\bx(t), \mathcal{S}) \leq C_0\cdot e^{-\rho_0 t}, \quad \forall t\geq 0.
    \end{align}
\end{theorem}

\begin{proof} To facilitate the proof, 
     we make the following two equivalent transformations for the ALC dynamics (\ref{gra2}):
    \begin{itemize}
        \item [1)] By equation (\ref{sol2:lam2}), we formulate $\bw_\mathcal{L}$  as 
        \begin{align}\label{wLe}
            \bw_\mathcal{L} =  D_\mathcal{L}^{-1}\left(  -\bd_\mathcal{L}-A_\mathcal{L} \bP + \bP^{in}_\mathcal{L}
        \right) 
        \end{align}
        and substitute it in equations (\ref{gra:d2}) and (\ref{gra:P2}). 
        \item [2)] By Lagrange's Mean Value Theorem, we have
        \begin{align} \label{nabc}
            \nabla c(\bd) - \nabla c(\bd^*) = C(\bd) (\bd -\bd^*)
        \end{align}
        where $C(\bd):= \text{diag}( c^{\prime\prime}_i(\hat{d}_i) )_{i\in\mathcal{N}}$ with some $\hat{d}_i$ depending on the value of $\bd$. Due to Assumption \ref{ass:strong}, we further have $u I \preceq C(\bd) \preceq \ell I$.
    \end{itemize}
    
As a consequence, the ALC dynamics (\ref{gra2}) can be equivalently reformulated as the following matrix form
\begin{align} \label{matt}
    \begin{split}
        \dot{\bz} = \Xi \underbrace{\begin{bmatrix}
        - C(\bd) - F_1 & - F_2^\top & \bm{0} & I_o^\top & I\\
        -F_2 & - F_3 & \bm{0} & A_{\mathcal{G}}^\top & \bm{0}\\
        \bm{0} & \bm{0} & \bm{0}& \bm{0} & S\\
        -I_o & -A_{\mathcal{G}} & \bm{0} & -D_{\mathcal{G}} & \bm{0}\\
        -I & \bm{0} & -S & \bm{0} &\bm{0}
        \end{bmatrix}}_{:=W(\bd)} \begin{bmatrix}
        \bd - \bd^* \\ \bP -\bP^*\\ \bpsi -\bpsi^*\\ \bw_{\mathcal{G}}
   - \bw_{\mathcal{G}}^*\\ \bmu -\bmu^*
        \end{bmatrix}
    \end{split}
\end{align}
where $I$ and $\bm{0}$  denote the identity matrix and zero matrix with appropriate dimensions,
$\Xi := \text{blockdiag}(\Xi_d,\Xi_P,\Xi_\psi,\Xi_\omega,\Xi_\mu)$, and $F_3: =  A_{\mathcal{L}}^\top D_{\mathcal{L}}^{-1}A_{\mathcal{L}}$. Besides, we have 
\begin{align*}
I_o: = \begin{bmatrix} I& \bm{0} \end{bmatrix},  \  F_1: = \begin{bmatrix}
    \bm{0} & \bm{0} \\ \bm{0}& D_{\mathcal{L}}^{-1}
    \end{bmatrix},\   F_2: = \begin{bmatrix} 
    \bm{0} & A_{\mathcal{L}}^\top D_{\mathcal{L}}^{-1}
    \end{bmatrix}
\end{align*}
where the first component and second component 
correspond to generator buses $i \in \mathcal{G}$ and   load buses $i \in \mathcal{L}$, respectively.  

To prove the global exponential stability of the ALC dynamics (\ref{gra2}), we design the quadratic Lyapunov function $V(\bz)$ as
\begin{align} \label{lya}
    V(\bz) = (\bz-\bz^*)^\top Q (\bz-\bz^*)
\end{align}
where $Q$ is defined by
\begin{align} \label{dQ}
    Q := \begin{bmatrix}
    \alpha I & \bo & \bo & \bo & I\\
    \bo & \alpha U_AU_A^\top & \bo & A_\mathcal{G}^\top  & \bo\\
    \bo & \bo & \alpha U_S U_S^\top & \bo & -\beta S\\
    \bo & A_\mathcal{G} & \bo & \alpha I &  \bo\\
    I &  \bo & -\beta S & \bo& \alpha I
    \end{bmatrix}.
\end{align}
Here,  $\alpha$ is a sufficiently large positive number and  $\beta$ is a sufficiently small positive number. $U_A$ is the right-singular matrix of matrix $A$ with the \emph{compact} singular value decomposition 
\begin{align} \label{svd}
    A = V_A \Sigma_A U_A^\top
\end{align} and $\Sigma_A\succ 0$.
$U_S$ is the normalized matrix corresponding to the \emph{compact} eigen-decomposition of matrix $S$ with
\begin{align}\label{egd}
    S = U_S \Sigma_S U_S^\top
\end{align}
and $\Sigma_S\succ 0$.
Thus we have 
\begin{align} \label{ppa}
    A_\mathcal{G}U_A U_A^\top = A_\mathcal{G},\     A_\mathcal{L}U_A U_A^\top = A_\mathcal{L}, \ S U_S U_S^\top = S.
\end{align}

Then we obtain the following two key lemmas, whose proofs are provided in Appendix \ref{pf-psd} and \ref{pf-lem}, respectively.

\begin{lemma} \label{le-psd}
  Matrix  $Q$ is positive semi-definite, i.e.,  $V(\bz)\geq 0$ for any $\bz$, and the set  
  \begin{align}\label{EPSS}
  \begin{split}
    \mathcal{M}&:=\left\{\hat{\bz}\,|\,V(\hat{\bz})\equiv 0\right\} \ \overset{\Delta}{=} \    \left\{\hat{\bz}\,|\,    \hat{\bd} = \bd^*, \hat{\bw}_\mathcal{G} = \bw_\mathcal{G}^*,  \right.\\
	&\qquad\qquad   \phantom{=\;\;}\left. \hat{\bmu} =\bmu^*, A \hat{\bP} = A\bP^*, S\hat{\bpsi} = S \bpsi^*\right\}.
  \end{split}
  \end{align}
\end{lemma}

\begin{lemma} \label{le-exp}
 Under Assumption \ref{ass:strong},   the time derivative of $V(\bz)$ along  the ALC dynamics (\ref{gra2}) satisfies that for $ \rho = \frac{\beta^2}{\alpha}>0$,
\begin{align} \label{upro}
\qquad  \frac{d V(\bz)}{d t} \leq -\rho V(\bz), \quad \forall t\geq 0.
\end{align}
\end{lemma}

By Lemma \ref{le-exp}, 
$ V(\bz(t))  \leq V(\bz(0))\cdot e^{-\rho t}$ for all $t\geq 0$. Decompose $\bz(t) - \bz^* = \bm{\delta}_1(t) + \bm{\delta}_2(t)$ such that $\bm{\delta}_1(t)\in \text{row}(Q)$ and $\bm{\delta}_2(t)\in \text{ker}(Q)$. Thus 
$V(\bz) = \bm{\delta}_1(t)^\top Q \bm{\delta}_1(t)$ and
\begin{align}
\begin{split}
    \text{dist}(\bx(t), \mathcal{S})  = &\inf_{\hat{\bx}(t)\in \mathcal{S}} ||\, [\bz(t); \bw_\mathcal{L}(t)] - [\hat{\bz}(t); \hat{\bw}_\mathcal{L}(t)]\,||\\
 \leq & \inf_{\hat{\bz}(t)\in \mathcal{M}} \sqrt{||T||^2+1}\cdot ||\bz(t) - \hat{\bz}(t)||\\
  \leq & \sqrt{||T||^2+1}\cdot ||\bz(t) - (\bz^*+\bm{\delta}_2(t)) ||\\
   \leq & \sqrt{ \frac{||T||^2+1}{\lambda_{\min}(Q)}    }\cdot ||\bm{\delta}_1(t) ||_Q \\
      \leq & \sqrt{ \frac{(||T||^2+1)\cdot V(\bz(0))}{\lambda_{\min}(Q)}    }\cdot \exp(-\frac{1}{2}\rho t)
\end{split}
\end{align}
where the first inequality is due to (\ref{wLe}) and let $\bw_\mathcal{L} - \hat{\bw}_\mathcal{L} = T (\bz - \hat{\bz})$ with corresponding matrix $T$. The second inequality is because $\bz^*+\bm{\delta}_2(t) \in \mathcal{M}$. For the third inequality, $\lambda_{\min}(Q)$ is the smallest positive eigenvalue of $Q$.

By taking $C_0: = \sqrt{ \frac{(||T||^2+1) V(\bz(0))}{\lambda_{\min}(Q)}    }$ and $\rho_0 := \rho/2$, 
 Theorem \ref{thm3} is proved. 
\end{proof}

\begin{remark} (Uniqueness of Equilibrium Point)  Lemma \ref{le-psd} indicates that the optimal $\bP^*$ and $\bpsi^*$ to the OLC problem (\ref{op2}) are not unique. The former is because the node-branch incidence matrix $A$ may not be of full column rank for a meshed network. The latter is caused because the (virtual) phase angle $\psi$ is defined in a relative reference frame in the power system without a slack bus, thus 
$A$ (or $\bar{A}$) is not of full row rank. Nevertheless, according to Theorem \ref{thm1}, the ALC dynamics (\ref{gra2}) eventually asymptotically converge to an equilibrium point which depends on the initial condition.
\end{remark}

\begin{remark} (Inequality Constraints) One natural question to ask about Theorem \ref{thm3} is whether the ALC dynamics
can still achieve global exponential convergence  when considering the inequality capacity constraints (\ref{olc:power_limits}, \ref{olc:thermal_limits}), or  (\ref{op:power_limits}, \ref{op:thermal_limits}). The  key challenge  is that the complete ALC algorithm (\ref{cm}) involves a discontinuous projection step, which creates difficulty in theoretical analysis. Actually, this question can be generalized as the problem  whether the standard projected primal-dual gradient dynamics (PDGD) is exponentially stable. In \cite[Remark 2]{guan}, it is conjectured that the PDGD with projection may not be exponentially stable due to the norm issue. Instead,  reference \cite{guan} proposes  a new PDGD using an augmented Lagrangian  to deal with the inequality constraints and proves it to be exponentially stable. Therefore, one of the future work is to leverage the augmented Lagrangian to design a distributed load frequency control algorithm with global exponential convergence. 
\end{remark}


\subsection{Dynamical Tracking Performance and Robustness} \label{sec-dyna}

In practice, the uncontrollable power injection $\bP^{in}$ is not a fixed value (i.e., step change) but time-varying due to the intrinsic volatility of renewable generation and load demand. Besides, the real implementation of the ALC algorithm suffers from 1) the  measurement and communication noises, 2) the model errors due to the use of DC power flow (\ref{DC}) and linear network dynamics (\ref{dynamic}). 
Hence, we study the dynamical tracking performance of the 
ALC dynamics (\ref{gra2}) in practical application 
 by leveraging its global exponential convergence.

Let $\bP^{in}(t)$ be the uncontrollable power injection at time $t$. Substituting it to the reduced ALC dynamics (\ref{gra2}), we can formulate the ALC dynamics under time-varying $\bP^{in}(t)$ as
\begin{align} \label{insdym}
    \dot{\bz} = \bm{f}(\bz) + H\bP^{in}(t)
\end{align}
with corresponding constant matrix $H$ and function $\bm{f}$.
Let $\bz^*(t) $ be an associated equilibrium point of dynamics (\ref{insdym}) given $\bP^{in}(t)$. Moreover, $\bz^*(t) $ is also a saddle point for the Lagrangian function of problem ~(\ref{op2}) under the uncontrollable power injection $\bP^{in}(t)$.

Taking time-varying power change, measurement noise and  model error into consideration,  the  actual load control dynamics  can be formulated as 
\begin{align}\label{pertu}
    \dot{\bz} = \bm{f}(\bz)+H\bP^{in}(t)+ \bm{g}(\bz,t)
\end{align}
where $\bm{g}(\bz,t)$ captures the real-time measurement and communication noise, model error and other potential mismatches.


We make the following standard assumption on bounded system mismatch and drift rate \cite{Hao, steven}.


\begin{assumption}\label{ass-changerate}
  The time-varying  equilibrium point $\bz^*(t)$ is differentiable and has a bounded drift rate in the sense that 
   there exists a positive constant $b_z$ such that  
         \begin{align} 
       ||\frac{d \bz^*(t)}{dt}  ||_Q &\leq b_z,\ \qquad \forall t\geq 0. \label{driz}
  \end{align}
In addition, the mismatch term $\bm{g}(\bz,t)$ in (\ref{pertu}) is bounded, i.e., there exists a positive constant $b_g$ such that  
    \begin{align}
        || \bm{g}(\bz,t)||_Q \leq b_g, \qquad \ \forall t\geq 0.
    \end{align}
 \end{assumption}



 Then the dynamical tracking properties under the actual load control dynamics (\ref{pertu}) are established as the following theorem. 
\begin{theorem} \label{thm-track}
Under Assumption \ref{assumption:feasibility}, \ref{ass:strong} and \ref{ass-changerate}, 
the tracking error of the actual load control dynamics (\ref{pertu}) is bounded in the sense that, for any time $ t\geq 0$,
 \begin{align} \label{trae}
    \begin{split}
           || \bz(t) -\bz^*(t)||_Q
     \leq  & \exp(-\frac{\rho}{2}t)\cdot ||\bz(0) - \bz^*(0)||_Q\\
      & + \left( 1-\exp(-\frac{\rho}{2}t)\right)\frac{2(b_z+b_g)}{\rho} 
    \end{split}
\end{align}
where
$Q$ and $\rho$ are given in  (\ref{dQ}) and (\ref{upro}) respectively.
\end{theorem}
\begin{proof}
By constraint (\ref{op2:Y}), we have $
\bP^{in}(t) = \bd^*(t) + S \bpsi^*(t)$ for all $t$. According to the expansion in (\ref{eq:expV}),  we obtain
\begin{align}\label{eq:bouPin}
   &\quad  || \dot{ \bP}^{in}(t) || \leq  ||\dot{\bd}^*(t)|| + ||S \dot{\bpsi}^*(t) || \nonumber\\
    & \leq \sqrt{ 2\left(||\dot{\bd}^*(t)||^2 + || U_S\Sigma_S||^2\cdot ||U_S^\top \dot{\bpsi}^*(t) ||^2 \right)}\nonumber \\
    & \leq \sqrt{\frac{2}{\hat{\alpha}}}\cdot \sqrt{ (\alpha -1)||\dot{\bd}^*(t)||^2 + (\alpha-\beta^2||\Sigma_S||^2) ||U_S^\top \dot{\bpsi}^*(t) ||^2 }\nonumber\\
    & \leq \sqrt{\frac{2}{\hat{\alpha}}}\cdot \sqrt{V(\dot{\bz}^*(t))} = \sqrt{\frac{2}{\hat{\alpha}}}\cdot ||\dot{\bz}^*(t)||_Q \leq \sqrt{\frac{2}{\hat{\alpha}}}\cdot b_z
\end{align}
where $\hat{\alpha}: = \min\{\alpha-1, {{(\alpha-\beta^2||\Sigma_S||^2})/||U_S\Sigma_S||^2}\}$.

Given an infinitesimal time step $\Delta>0$, we consider the time period  $[m\Delta, (m+1)\Delta]$ where $m$ is a non-negative integer. Let  $\bz(t)$ be the state variable following the  real system dynamics (\ref{pertu}), while we denote $\hat{\bz}(t)$ for $t\in[m\Delta,(m+1)\Delta]$ as the state following the ALC dynamics (\ref{insdym}) with fixed power injection $\bP^{in}({m\Delta})$, and
  $\hat{\bz}(m\Delta) = {\bz}(m\Delta)$. For notational simplicity,  denote $\bz_m := \bz(m\Delta)$, which is similar for $\hat{\bz}_m$.
  
For any $t\in[m\Delta,(m+1)\Delta]$, 
 \begin{align}
   &  ||\bz(t) - \hat{\bz}(t)||_Q \nonumber \\ \leq & \int_{m\Delta}^t \left[ \sup_{\tau \in [m\Delta,(m+1)\Delta]}|| f(\bz(\tau)) -f(\hat{\bz}(\tau))  ||_Q \right]\, ds \nonumber\\ 
     & + \int_{m\Delta}^t ||H(\bP^{in}(s) -\bP^{in}(m\Delta)) + \bm{g}(\bz,s) ||_Q\,ds \nonumber \\
     \leq &    \left[  \sup_{\tau \in [m\Delta,(m+1)\Delta]}||  \bz(\tau) -\hat{\bz}(\tau)||_Q\right]\cdot \ell_f\cdot \Delta \nonumber\\
     &+ \int_{m\Delta}^{(m+1)\Delta} ||H\int_{m\Delta}^{s} (\frac{d \bP^{in}(\tau)}{d \tau})\  d \tau ||_Q \,ds \nonumber\\
     & + \int_{m\Delta}^{(m+1)\Delta} || \bm{g}(\bz,s)  ||_Q \,ds \nonumber \\
 \leq &    \left[  \sup_{\tau \in [m\Delta,(m+1)\Delta]}||  \bz(\tau) -\hat{\bz}(\tau)||_Q\right]\cdot \ell_f\cdot \Delta \nonumber  \\
 & + ( b_P\Delta+ b_g )\cdot\Delta \label{eq:diffzhat}
 \end{align} 
 where ${b}_P:= \frac{1}{2}\sqrt{\frac{2}{\hat{\alpha}}} b_z\cdot||Q^{1/2}H||$, and $\ell_f$ is the Lipschitz constant for function $f(\bz)$ (it can be checked that $f(\bz)$ is Lipschitz continuous with bounded $\ell_f$). Since (\ref{eq:diffzhat}) holds for any $t\in[m\Delta,(m+1)\Delta]$, we obtain
 \begin{align} \label{eq:supdffz}
     \sup_{t \in [m\Delta,(m+1)\Delta]}||  \bz(t) -\hat{\bz}(t)||_Q \leq \frac{(b_P\Delta+ b_g)\cdot\Delta }{1-\ell_f\cdot\Delta} 
 \end{align}

  For the tracking error at time $(m+1)\Delta$, we have
\begin{align} \label{traerr}
     &\quad  ||\bz_{m+1} - \bz^*_{m+1}||_Q \nonumber\\
        &\leq ||\bz^*_{m+1} - \bz^*_m||_Q +||\bz_{m+1} - \bz^*_m||_Q \nonumber \\
&\leq  || \int_{m\Delta}^{(m+1)\Delta} (\frac{d\bz^*(t)}{dt}) dt  ||_Q \nonumber\\
&\quad +  ||\hat{\bz}_{m+1} -\bz^*_m||_Q +   ||\bz_{m+1} -\hat{\bz}_{m+1}||_Q \nonumber\\
       & \leq  \  b_z\cdot \Delta + e^{-\frac{\rho\Delta}{2}}\cdot ||\bz_m - \bz^*_m||_Q + \frac{(b_P\Delta+ b_g)\Delta }{1-\ell_f\Delta} 
\end{align}
where the third inequality  results from Lemma \ref{le-exp} and (\ref{eq:supdffz}).

Let $t=m\Delta$. Using inequality (\ref{traerr}) recursively, we obtain
 \begin{align}\label{zt-}
    \begin{split}
      & || \bz(t) -\bz^*(t)||_Q \ 
      \leq  (e^{-\frac{\rho\Delta}{2}})^m\cdot ||\bz(0)-\bz^*(0)||_Q \\
      &\qquad\qquad+\frac{1-(e^{-\frac{\rho\Delta}{2}})^m}{1-e^{-\frac{\rho\Delta}{2}}}(b_z+  \frac{b_P\Delta+ b_g }{1-\ell_f\Delta})\cdot\Delta \\
      =&\   e^{-\frac{\rho }{2}t} \cdot||\bz(0)-\bz^*(0)||_Q \\
      &\qquad\qquad+\frac{ (1-e^{-\frac{\rho t}{2}})\cdot\Delta}{1-e^{-\frac{\rho \Delta}{2}}}(b_z+  \frac{b_P\Delta+ b_g }{1-\ell_f\Delta}).
    \end{split}
\end{align}
Note that inequality (\ref{zt-}) holds for any $\Delta\geq 0$ and non-negative integer $m$. Hence, let $\Delta\rightarrow 0$, using 
L'H\^opital's rule and standard arguments of Calculus, we obtain  inequality (\ref{trae}).

\end{proof}

Theorem \ref{thm-track} shows that the dynamical tracking errors of the system frequency and the load control schemes, i.e., $||\bw(t)-\bw^*||$ and $||\bd(t)-\bd^*(t)||$, are bounded. It is straightforward to see that as $t\rightarrow \infty$, the steady-state tracking error of the actual ALC dynamics (\ref{pertu}) is bounded by 
 \begin{align}\label{sted}
    \begin{split}
        \limsup_{t\to +\infty}   || \bz(t) -\bz^*(t)||_Q
     \leq   \frac{2(b_z+b_g)}{\rho}. 
    \end{split}
\end{align}

 Essentially, the drift of the equilibrium points $\bz^*(t)$ is caused by the time-varying uncontrollable power injection $\bP^{in}(t)$. It can be proved that the drift rate of the equilibrium points $\bz^*(t)$ is bounded with a bounded drift rate of $\bP^{in}(t)$, which is restated as the following proposition.
 
 \begin{proposition} \label{pro:bound}
Under Assumption \ref{assumption:feasibility} and \ref{ass:strong}, if the time-varying uncontrollable power injection $\bP^{in}(t)$ and the equilibrium points $\bz^*(t)$ are time differentiable, there exists a positive constant $\eta$ such that 
 \begin{align*}
    ||\frac{d \bz^*(t)}{dt}||_Q\leq \eta \cdot \sup_{t\geq 0}||\frac{d \bP^{in}(t)}{dt}||, \quad \forall t\geq 0.
 \end{align*}
 \end{proposition}

The proof of Proposition \ref{pro:bound} is provided in Appendix \ref{app:pro}.


\section{Case Studies}\label{sec:simulations}

The effectiveness and robustness of the proposed ALC algorithm are demonstrated in numerical simulations. In particular, the performance of the ALC algorithm under step and continuous power changes is tested, and the cases with inaccurate damping coefficients are demonstrated. The impact of noises in measurements is also studied numerically.

\subsection{Simulation Setup}
The 39-bus New England power network in Figure \ref{39bus} is used as the test system. The simulations were run on Power System Toolbox (PST) \cite{pst1}, and we embedded the proposed ALC algorithm (\ref{cm}) through modifying the dynamic model functions of PST. Compared to the analytic model \eqref{dynamic}, the PST simulation models are more complicated and realistic, which involve the classic two-axis subtransient generator model, the IEEE Type DC1 excitation system model, the alternating current (AC) power flow model, and different types of load models. Detailed configuration and parameters of the simulation model are available online \cite{pst2}.  

\begin{figure}
	\centering
	\includegraphics[scale=0.43]{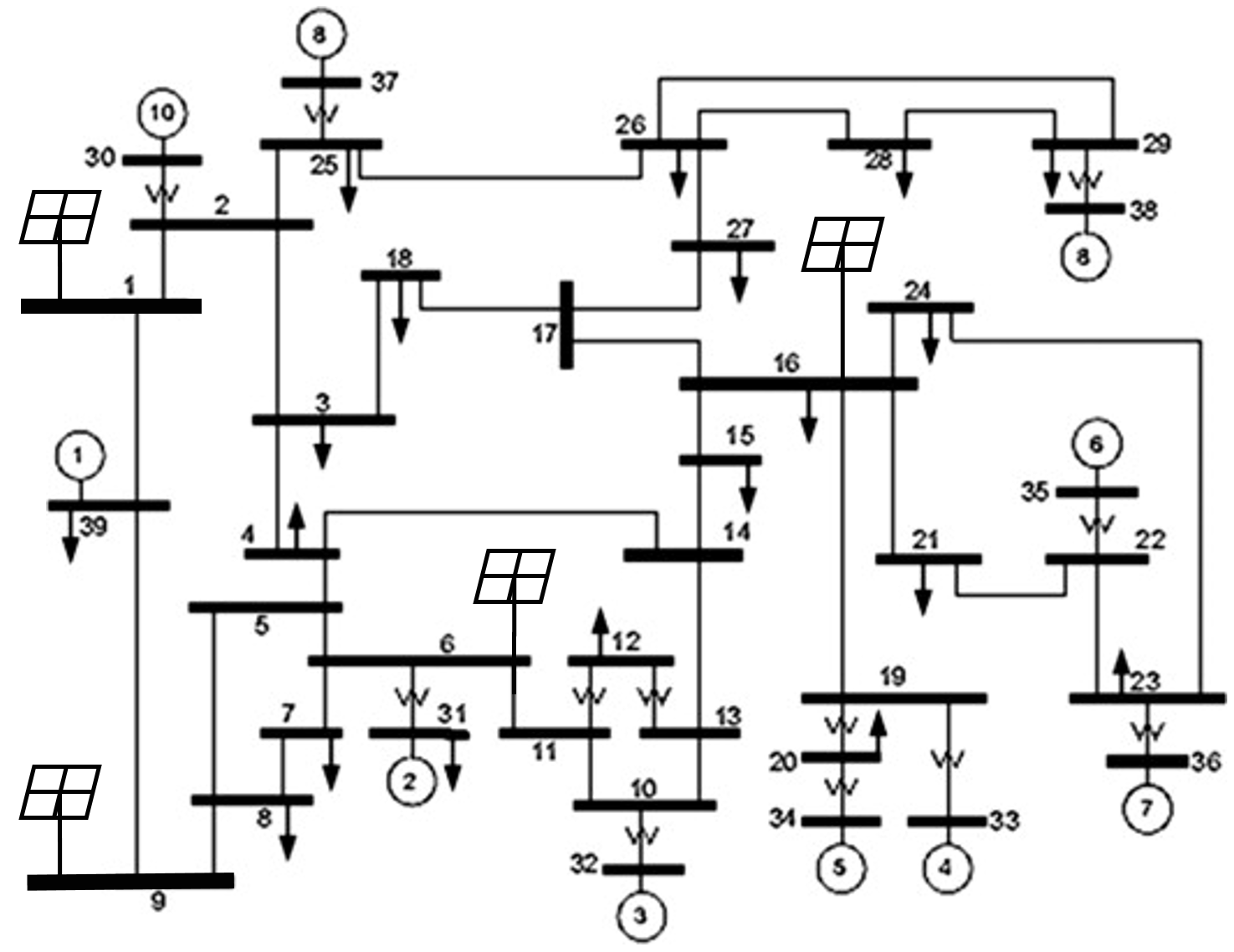}
	\caption{The 39-bus New England power network.}
	\label{39bus}
\end{figure}

There are 11 generators located at bus-29 to bus-39, which are the generator buses. To simulate continuous changes in power supply, four photovoltaic (PV) units are added to bus-1, bus-6, bus-9, and bus-16. Since PV units are integrated with  power electronic interfaces, we regard them as negative loads rather than swing generators. Consequently, bus-1 to bus-28 are load buses with a total active power demand of 6.2 GW. Every load bus has an aggregate controllable load, and the disutility function for load control is 
$$c_i\left(d_i\right)=\vartheta_i\cdot d_i^2$$
where the cost coefficients $\vartheta_i$ are set to 1 per unit (p.u.) for bus-1 to bus-5, and 5 p.u. for other load buses. The adjustable load limits are set as $\overline{d}_i=-\underline{d}_i=0.4$ p.u. with the base power being 100 MVA. In addition, the loads are controlled every 250 ms, which is a realistic estimate of the time-resolution for load control \cite{rea}. The damping coefficient $D_i$ of each bus is set to 1 p.u. For the load controller, the step sizes $\epsilon$ and the constants $K_i$ are all set to 0.5 p.u.

\subsection{Step Power Change}

At time $t=1$ s, step load increases of 1 p.u.  occurred at bus-1, bus-6, bus-9, and bus-16. With or without ALC, the system frequency is illustrated in Figure \ref{ALC_fre}. 
It is observed that the power network itself is not capable of restoring the nominal system frequency without ALC. 
In contrast, the proposed ALC scheme can bring the system frequency back to the nominal value. Figures \ref{load_c} presents the load adjustments and the total cost of load control under ALC, respectively. It is seen that the loads with lower cost coefficients $\vartheta_i$ tend to make larger adjustments, which are bounded by the capacity limits. 
 This observation indicates that the load adjustments are computed  to achieve system-wide efficiency although the control decisions are made locally. As a result, the total cost of the ALC scheme converges to the optimal cost of the OLC problem \eqref{olc} or \eqref{op} in the steady state. 

\begin{figure}[thpb]
		\raggedleft
	\includegraphics[scale=0.39]{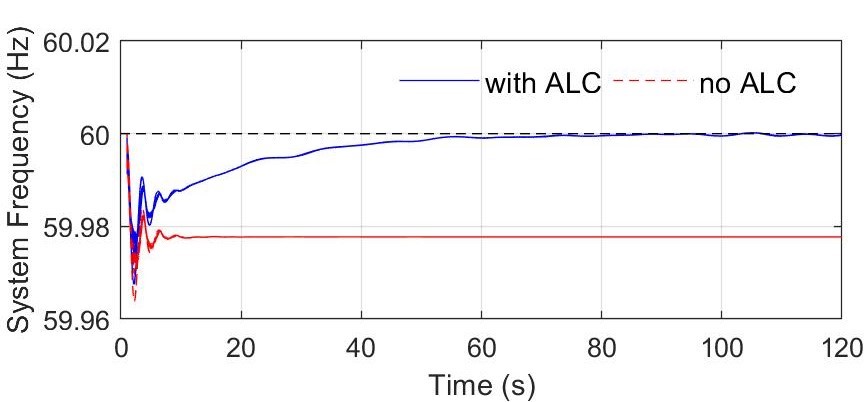} 
	\caption{The frequency dynamics under step power changes.}
	\label{ALC_fre}
\end{figure}

\begin{figure}[thpb]
	\raggedleft
		\includegraphics[scale=0.39]{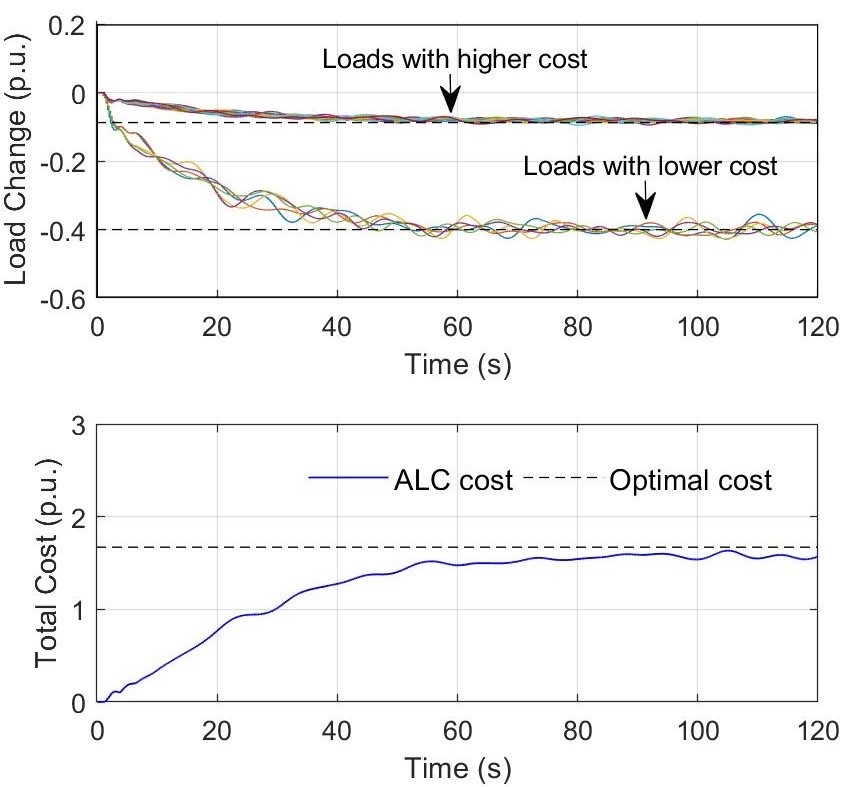}
	\caption{The load adjustment scheme and the total ALC cost.}
	\label{load_c}
\end{figure}

\subsection{Continuous Power Change}\label{ca-con}

We next study the performance of ALC under continuous power changes. 
To this end, the PV generation profiles of a real power system located within the territory of Southern California Edison are utilized as the
power outputs of the four PV units. 
The original 6-second data of PV outputs are linearly interpolated to generate power outputs every 0.01 second, which is consistent with the resolution of PST dynamical simulation. The PV power outputs over 10 minutes are shown in Figure \ref{pv}.  Figures \ref{fre_pv} and \ref{vol} illustrate the dynamics of system frequency and voltage magnitudes, respectively.

\begin{figure}[thpb]
	\centering
	\includegraphics[scale=0.33]{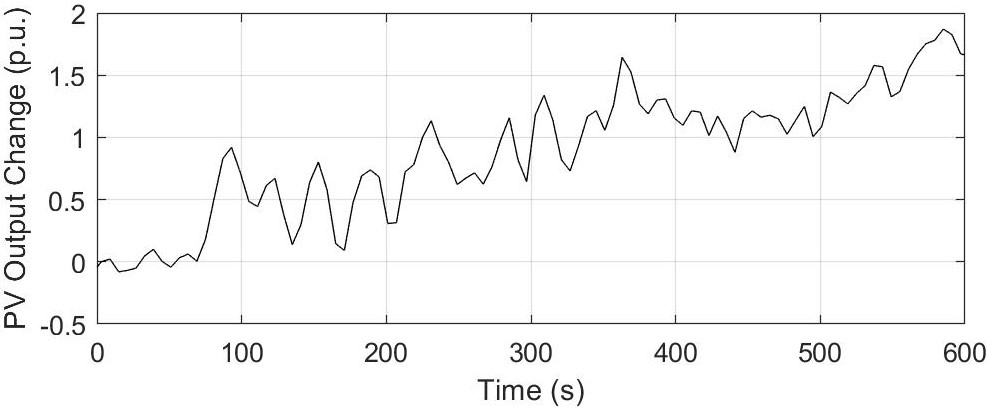}
	\caption{The PV power outputs.}
	\label{pv}
\end{figure}

\begin{figure}[thpb]
	\centering
	\includegraphics[scale=0.24]{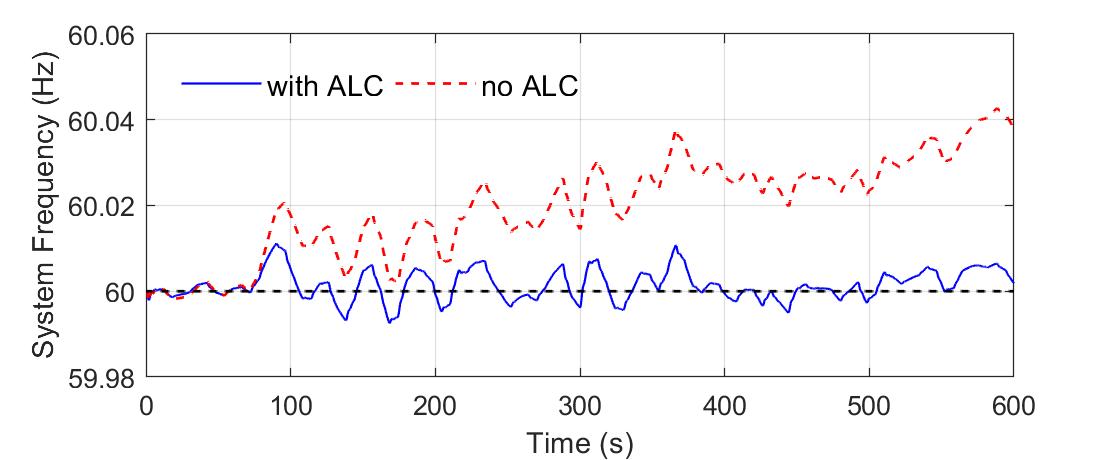}
	\caption{The frequency dynamics under continuous power changes.}
	\label{fre_pv}
\end{figure}

\begin{figure}[thpb]
	\centering
	\includegraphics[scale=0.305]{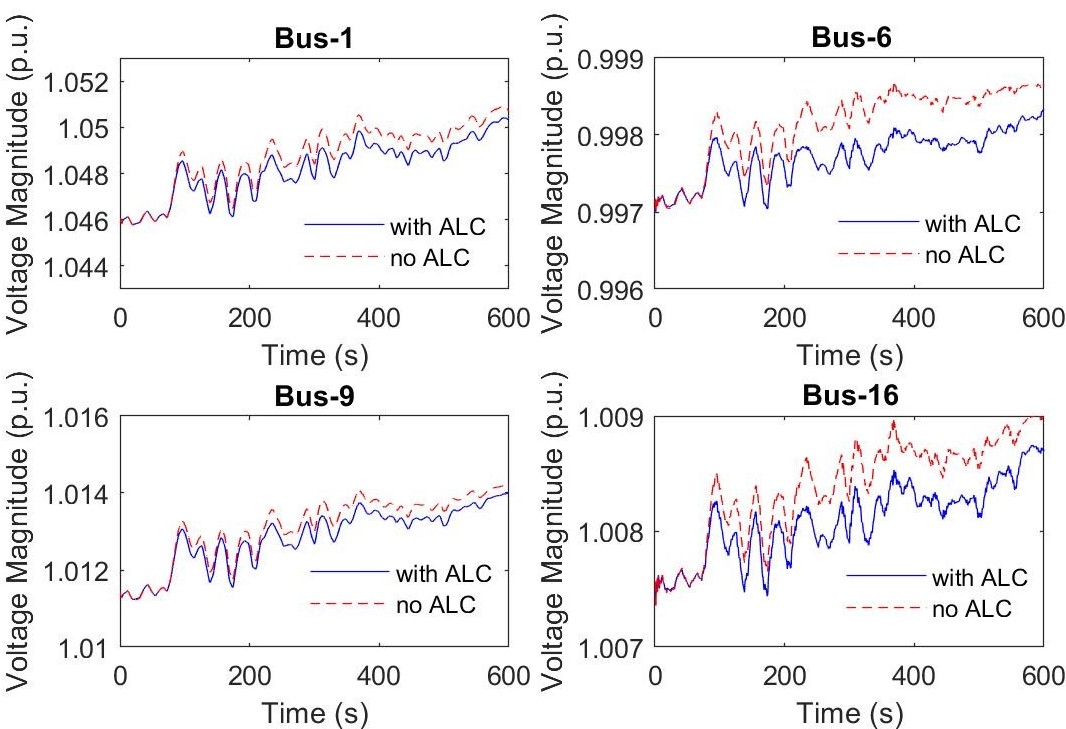}
	\caption{The dynamics of voltage magnitudes at the PV buses.}
	\label{vol}
\end{figure}
 
 From Figure \ref{fre_pv}, it is seen that ALC can effectively maintain the system frequency around the nominal value under 
 time-varying power imbalance. Although the proposed ALC algorithm is designed for step power changes, it can handle the case with continuous power disturbance due to the utilization of real-time frequency and power flow information. \add{Besides, from Figure \ref{vol}, it is observed that the voltage rise caused by increased PV generation is alleviated with the ALC scheme. 
 The reason is that the power imbalance is eliminated by the coordinated adjustment of many
ubiquitously distributed loads when using the ALC scheme, instead of the generation control of few generators, thus 
 the voltage rise (or descent) along with the power flow is mitigated. From the simulation results, the ALC scheme not only can maintain system frequency, but also may improve the dynamics of voltage magnitudes. 
 }

\subsection{Impact of Inaccurate Damping Coefficients}

This part is devoted to understanding the impact of inaccurate damping coefficients on the performance of ALC. Let the damping coefficient $\tilde{D}$ used by the controller be $k$ times of the accurate value ${D}$ with $\tilde{D}_i=k\cdot{D}_i$ for each bus $i\in\mathcal{N}$. Then we tuned the factor $k$ to test the performance of ALC under step power changes. Figure \ref{damp} compares the frequency dynamics using the ALC scheme with different $k$. 

\begin{figure}[thpb]
		\centering
	\includegraphics[scale=0.33]{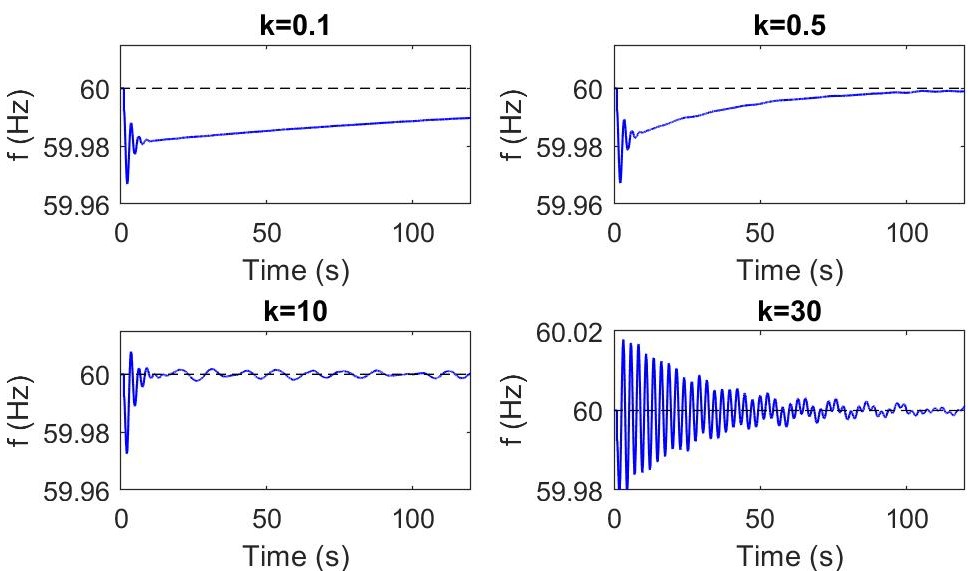}
	\caption{The frequency dynamics under inaccurate damping coefficients.}
	\label{damp}
\end{figure}

As shown in Figure \ref{damp}, the convergence of system frequency becomes slower when smaller damping coefficients are used. 
As the utilized damping coefficients  approach zero, ALC can still stabilize the system frequency but can not restore the nominal value. That is because when $D=0$, the OLC problem \eqref{op} imposes no restriction on the system frequency. As a result, 
only the power imbalance is eliminated, but the nominal frequency cannot be restored. 
In contrast, when larger damping coefficients are utilized, the convergence of frequency dynamics becomes faster, at the cost of increased oscillations. Generally, the ALC scheme can work well 
 under moderate inaccuracies in the damping coefficients $D$.

\subsection{Impact of Measurement Noise}

 Recall that the implementation of ALC requires 
the local measurement of frequency deviation $\omega_i$ and adjacent power flows $(P_{ki}, P_{ij})$ at each bus $i \in \mathcal{N}$. Therefore this part studies how the  measurement noises affect the performance of ALC.

First, consider the noise  $\xi_i^\omega$ in the measurement of $\omega_i$ and let the measured frequency deviation be $\tilde{\omega}_i=\omega_i+\xi_i^\omega$. Assume that the noise $\xi_i^\omega$  follows Gaussian distribution $\mathcal{N}(0,\,\sigma_\omega^{2})$, then  
the standard deviation $\sigma_\omega$ is tuned to test the performance of ALC under step power changes. In each simulation, 
 the noise $\xi_i^\omega$ is generated independently over time and across buses. The resultant frequency dynamics and load adjustment scheme are shown as Figure \ref{f_n}.

\begin{figure}[thpb]
	\centering
	\includegraphics[scale=0.07]{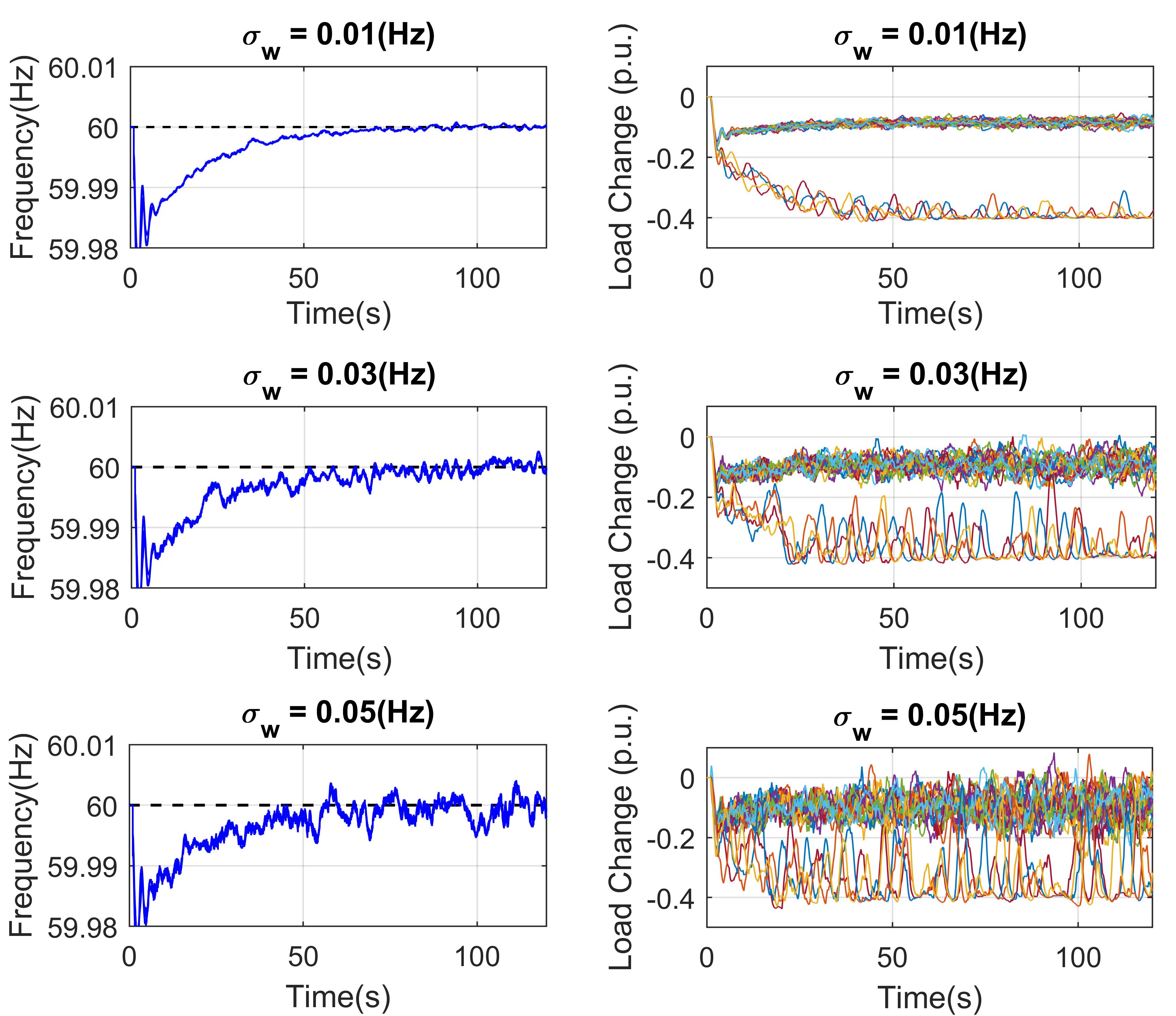}
	\caption{The frequency dynamics and load adjustment with different noises in frequency measurement.}
	\label{f_n}
\end{figure}

Then, we inject noise $\xi_{ij}^P$ to the measurement of power flow with $\tilde{P}_{ij} = P_{ij}+ \xi_{ij}^P$ and assume that  $\xi_i^P \sim \mathcal{N}(0,\,\sigma_P^{2})$.
The frequency dynamics and load adjustment scheme  under different levels of power flow noises are shown in Figure \ref{f_p}.
\begin{figure}[thpb]
	\centering
	\includegraphics[scale=0.07] {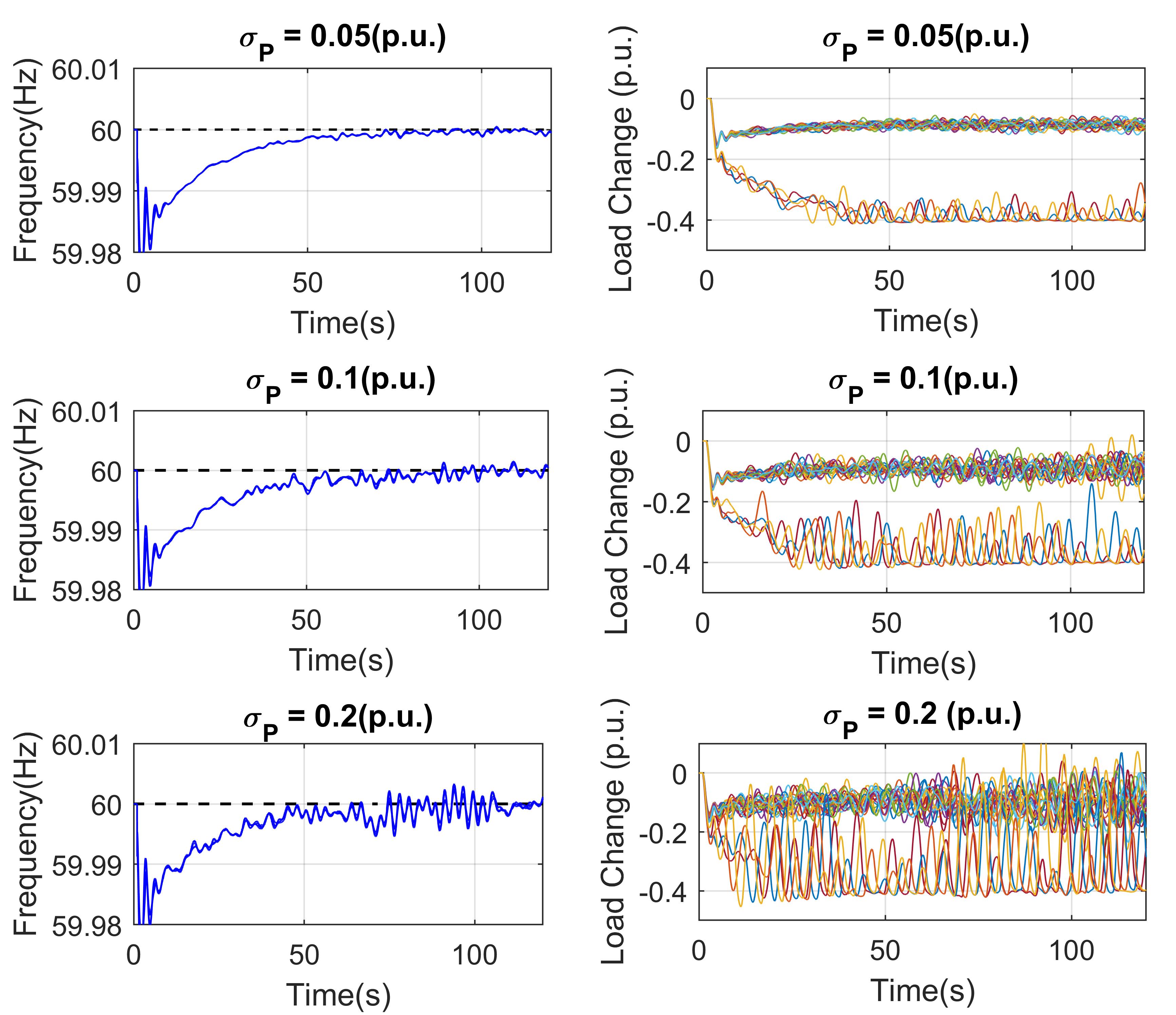}
	\caption{The frequency dynamics and load adjustment with different noises in power flow measurement.}
	\label{f_p}
\end{figure}

\add{
From Figures \ref{f_n} and \ref{f_p}, it is observed that the system frequency can be restored to the nominal value under the measurement noise, while the loads are continuously modulated in response to the measurement errors. Moreover, in both cases, higher level of noise leads to larger oscillations in the system frequency and greater fluctuations of the load adjustment.
}

\section{Conclusion}\label{sec:conclusion}

Based on the reverse engineering approach, we developed a fully distributed ALC mechanism for frequency regulation in power systems. The combination of ALC and power network dynamics was interpreted as a partial primal-dual gradient algorithm to solve an optimal load control problem. As a result, relying purely on local measurement and local communication, ALC can eliminate power imbalance and restore the nominal frequency with minimum total cost of load adjustment, while respecting operational constraints such as load power limits and line thermal limits. 
Numerical simulations of the 39-bus New England system showed that ALC can maintain system frequency under step or continuous power changes, and is robust to inaccuracy in damping coefficients as well as measurement noises.

\appendices

{\color{black}

\section{Proof of Lemma \ref{le-psd}} \label{pf-psd}

Define $\Dz : = \bz -\bz^* = [\Dd; \DP;\Dpsi, \Dw,\Dmu]$. Split the matrix  $V_A $ in (\ref{svd}) as $V_A =[V_A^{\mathcal{G}}; V_A^{\mathcal{L}}]$ where  $V_A^{\mathcal{G}}$ and  $V_A^{\mathcal{L}}$ respectively collect the rows w.r.t. the generator buses and the load buses. Then we have $A_\mathcal{G}= V_A^{\mathcal{G}} \Sigma_A U_A^\top $ by (\ref{svd}).

For any $\Dz$, we have 
\begin{align} \label{eq:expV}
    \begin{split}
    V(\bz) = &  \Dz^\top Q \Dz = \alpha  ||\Dd||^2 + \alpha||U_A^\top \DP||^2 \\
    &+ \alpha ||U_S^\top \Dpsi||^2 + \alpha ||\Dw||^2 +\alpha ||\Dmu||^2 \\
    &+ 2\Dd^\top \Dmu + 2 \Dw^\top A_\mathcal{G} \DP -2\beta \Dpsi^\top S \Dmu \\
    = & ||\Dd +\Dmu||^2 + ||\Dw +A_\mathcal{G}\DP ||^2 \\
    & + ||\Dmu -\beta  S\Dpsi||^2 + (\alpha-1)||\Dd||^2 \\
    & + (\alpha-2)||\Dmu||^2+(\alpha-1)||\Dw||^2 \\
    & +(U_S^\top \Dpsi)^\top \underbrace{(\alpha I - \beta^2 \Sigma_S^2)}_{:=Z_1} U_S^\top \Dpsi\\
    & + (U_A^\top \DP)^\top \underbrace{(  \alpha I - \Sigma_A {V_A^\mathcal{G}}^\top V_A^\mathcal{G} \Sigma_A)}_{:=Z_2}U_A^\top \DP
    \end{split}
\end{align}

Since parameter $\alpha$ is sufficiently large and parameter $\beta$ is sufficiently positively small, we have 
\begin{subequations}
\begin{align}
    Z_1&\succeq (\alpha  -\beta^2 ||\Sigma_S||^2) \cdot I \succ 0\\
    Z_2&\succeq (\alpha - ||V_A^\mathcal{G} \Sigma_A||^2)\cdot I \succ 0
\end{align}
\end{subequations}
Therefore, $V(\bz)\geq 0$ for any $\Dz$, i.e., $Q\succeq 0$. In addition, $V(\bz) = 0$ if and only if 
\begin{align*}
    \Dd = 0, \Dw =0, \Dmu =0, U_A^\top \DP =0, U_S^\top \Dpsi =0 
\end{align*}
It can be further checked that by equations (\ref{svd}) and (\ref{egd}), 
\begin{align*}
 &   \left\{(\DP,\Dpsi)\, |\, U_A^\top \DP =0, U_S^\top \Dpsi =0    \right\}  \\ 
  &\qquad\qquad   \overset{\Delta}{=} \ \left\{(\DP,\Dpsi)\, |\, A \DP =0, S \Dpsi =0    \right\} 
\end{align*}

\section{Proof of Lemma \ref{le-exp}} \label{pf-lem}

Without loss of generality, let the step size matrix $\Xi$ be the identity matrix $I$ for simplicity. 
Then the time derivative of $V(\bz)$ can be formulated as 
\begin{align}
    \begin{split}
        \frac{d V(\bz)}{d t} &= \dot{\bz}^\top Q (\bz-\bz^*) + (\bz-\bz^*)^\top Q\dot{\bz}\\
        & =  (\bz-\bz^*)^\top[W(\bd)^\top Q+QW(\bd)](\bz-\bz^*)
    \end{split}
\end{align}
Hence, it is sufficient to prove Lemma \ref{le-exp} by showing 
\begin{align}
   R(\bd):=  - W(\bd)^\top Q-QW(\bd)-\rho Q\succeq 0
\end{align}
for any  $\bd$. 

Plugging the definition of $Q$ (\ref{dQ}) and $\rho = \frac{\beta^2}{\alpha}$, we obtain
\begin{align} \label{Rd}
    \begin{split}
        R(\bd) 
      = \begin{bmatrix}
      L_d&  L_{Pd}^\top   & (\beta-1) S & F_2^\top  A_{\mathcal{G}}^\top &  L_{d\mu}\\
     L_{Pd}   & L_P& \bm{0} &  L_{\omega P}^\top & -F_2\\
        (\beta-1) S & \bm{0} & L_\psi & \bm{0} & -\frac{\beta^3}{\alpha} S\\
         A_{\mathcal{G}} F_2 &  L_{\omega P}& \bm{0} & L_\omega & I_o\\
         L_{d\mu} & -F_2^\top  & -\frac{\beta^3}{\alpha} S & I_o^\top & L_\mu
        \end{bmatrix}
    \end{split}
\end{align}
where 
\begin{subequations}
\begin{align}
    L_d & :=2 \alpha C(\bd) + 2\alpha F_1-2I - \beta^2 I 
    \label{Ld}\\
    L_P &:=  2\alpha  F_3+2  A_{\mathcal{G}}^\top  A_{\mathcal{G}} -\beta^2 U_A U_A^\top\\
    L_\psi &: =  2\beta SS -\beta^2 U_S U_S^\top\\
    L_\omega &: =2\alpha D_{\mathcal{G}}-2  A_{\mathcal{G}} A_{\mathcal{G}}^\top-\beta^2 I
    \\
    L_\mu & : = 2I-2\beta SS - \beta^2 I 
    \\
    L_{Pd}&:=  2\alpha  F_2 +  A_{\mathcal{G}}^\top I_o
    \\
    L_{d\mu} & :=  -C(\bd)-F_1 +{\beta^2}/{\alpha}\cdot I \\
    \begin{split}
    L_{\omega P} & : = D_\mathcal{G} A_{\mathcal{G}} +  A_{\mathcal{G}} F_3^\top -{\beta^2}/{\alpha}\cdot A_{\mathcal{G}} \\
    &\ = \underbrace{ \begin{bmatrix} D_\mathcal{G}-\frac{\beta^2}{\alpha} I & A_\mathcal{G} A_\mathcal{L}^\top D_\mathcal{L}^{-1}
              \end{bmatrix} }_{:=H_1} \begin{bmatrix}  A_\mathcal{G} \\A_\mathcal{L} 
              \end{bmatrix} =H_1A
    \end{split}
\end{align}
\end{subequations}

Some terms are cancelled out by using (\ref{ppa}) 
when deriving the formulation of $R(\bd)$ (\ref{Rd}). 
The key observation to show $R(\bd)\succeq 0$ is that $R(\bd)$ is almost diagonally dominant with positive (semi-)definite diagonal blocks when $\alpha$ is sufficiently large and $\beta$ is  positively small enough. 

For any vector $\be := [\be_d; \be_P; \be_\psi;\be_\omega; \be_\mu ]$ corresponding to the components of $R(\bd)$ (\ref{Rd}), the quadratic term $\be^\top R(\bd) \be$ is formulated as follows:
\begin{align} \label{drd}
    \begin{split}
      &  \be^\top R(\bd) \be = \be_d^\top L_d \be_d + \be_P^\top L_P \be_P + \be_\psi^\top L_\psi \be_\psi + \be_w^\top L_w \be_w \\
        &\qquad\quad + \be_\mu^\top L_\mu \be_\mu 
        +  2 \be_P^\top L_{Pd} \be_d  + 2\be_\omega^\top  A_{\mathcal{G}} F_2 \be_d\\
        &\qquad\quad  +2\be_\omega^\top L_{\omega P}  \be_P+ 2(\beta-1) \be_\psi^\top S \be_d + 2\be_\mu^\top L_{d\mu} \be_d\\
        &\qquad \quad - 2\be_P^\top F_2 \be_\mu - 2{\beta^3}/{\alpha}\cdot \be_\mu^\top S\be_\psi + 2\be_w^\top I_o \be_\mu\\
        & = \be_\psi^\top T_\psi \be_\psi + \be_w^\top T_w \be_w+ \be_\mu^\top T_\mu \be_\mu + || \be_d + F_2^\top  A_{\mathcal{G}}^\top \be_\omega  ||^2  \\ 
        &\quad + || \beta U_A^\top \be_P + \frac{1}{\beta}\Sigma_A V_A^\top H_1^\top \be_\omega  ||^2+ ||\frac{\beta-1}{\beta} \be_d + \beta S\be_\psi||^2\\
        &\quad + || \frac{1}{2}\be_\mu + 2L_{d\mu} \be_d  ||^2 +|| \frac{1}{2}\be_\mu - 2F_2^\top \be_P    ||^2
        \\
        &\quad +|| \frac{1}{2}\be_\mu - 2\frac{\beta^3}{\alpha} S \be_\psi    ||^2+
        || \frac{1}{2}\be_\mu +2 I_o^\top \be_\omega  ||^2 \\
        &\quad + 2\alpha \begin{bmatrix}
        \be_d^\top & \be_P^\top
        \end{bmatrix} \underbrace{\begin{bmatrix}
        T_d & F_2^\top +\frac{1}{2\alpha} I_o^\top  A_{\mathcal{G}}   \\  F_2+\frac{1}{2\alpha}  A_{\mathcal{G}}^\top I_o   & T_P
        \end{bmatrix}}_{:= H_2(\bd)} \begin{bmatrix}
        \be_d \\ \be_P
        \end{bmatrix} 
    \end{split}
\end{align}
where 
\begin{subequations}
\begin{align}
    T_d &: = \frac{1}{2\alpha}L_d -\frac{(\beta-1)^2}{2\alpha\beta^2}I-
    \frac{1}{2\alpha}I-\frac{2}{\alpha} L_{d\mu}L_{d\mu}\\
     T_P&: =  \frac{1}{2\alpha} L_P-\frac{\beta^2}{2\alpha}U_AU_A^\top-\frac{2}{\alpha}F_2F_2^\top \\
      T_\psi&: = L_\psi-\beta^2SS -\frac{4\beta^6}{\alpha^2}SS\\ 
      T_\omega&: = L_\omega - \underbrace{\frac{1}{\beta^2} H_1 V_A\Sigma_A^2 V_A^\top H_1^\top-A_{\mathcal{G}}F_2 F_2^\top  A_{\mathcal{G}}^\top -4 I_o I_o^\top}_{:=H_3} \\
       T_\mu&: = L_\mu -I
\end{align}
\end{subequations}

For equation (\ref{drd}), when parameter $\alpha$ is sufficiently large and parameter $\beta$ is positively sufficiently small, we have

\noindent 1) $T_\psi \succeq 0$ because 
    \begin{align*}
        T_\psi & = \beta \left[(2-\beta -\frac{4\beta^5}{\alpha^2}) SS-\beta U_S U_S^\top   \right] \\
        & \succeq \beta U_S {\left[ (2-\beta -\frac{4\beta^5}{\alpha^2}) (\sigma_S^{\min})^2-\beta  \right]} U_S^\top\succeq 0
    \end{align*}
    where $\sigma_S^{\min}$ is the smallest positive eigenvalue of $S$.

\noindent 2) $T_\omega\succeq 0$ because $||H_3||$ is bounded when $\beta>0$ is small and fixed, and thus
    \begin{align*}
        T_\omega &\succeq   \left(2\alpha\cdot  \min_{i\in\mathcal{G}}\{D_i\}- 2||A_{\mathcal{G}}A_{\mathcal{G}}^\top||-\beta^2 - ||H_3||\right)\cdot I \\
        & \succeq 0\quad (\text{when $\alpha$ is sufficiently large})
    \end{align*}
    
\noindent 3) $T_\mu \succeq \left(1 -2\beta ||S||^2-\beta^2\right)\cdot  I\succeq 0$. 
   
\noindent 4) We further claim that $H_2(\bd)\succeq 0$ for any $\bd\in\mathbb{R}^{|\mathcal{N}|}$. This can be shown by using the Schur Complement Theorem. By Assumption \ref{ass:strong}, we have
    \begin{align*}
    \begin{split}
                T_d&\succeq  \left( u-\frac{3+\beta^2}{2\alpha}
        -\frac{(\beta-1)^2}{2\alpha\beta^2}-\frac{2}{\alpha}||L_{d\mu} ||^2
        \right)I  +F_1 \\
        &\succeq \frac{1}{2}u I + F_1 \succ 0
    \end{split}
    \end{align*}
   Then consider the  Schur complement of the block $T_P$ in $H_2(\bd)$, which is 
    \begin{align*}
        & T_P - (F_2+\frac{1}{2\alpha}A_{\mathcal{G}}^\top I_o )\, T_d^{-1} (F_2^\top+ \frac{1}{2\alpha}I_o^\top A_{\mathcal{G}} )\\
         \succeq& A_{\mathcal{L}}^\top D_{\mathcal{L}}^{-1}A_{\mathcal{L}} +\frac{1}{\alpha}A_{\mathcal{G}}^\top A_{\mathcal{G}} -\frac{\beta^2}{\alpha}U_A U_A^\top -\frac{2}{\alpha} F_2F_2^\top \\ &
        - (F_2+\frac{1}{2\alpha}A_{\mathcal{G}}^\top I_o )( \frac{u}{2} I +F_1  )^{-1} (F_2^\top+ \frac{1}{2\alpha}I_o^\top A_{\mathcal{G}} )\\
     =&  A_{\mathcal{L}}^\top D_{\mathcal{L}}^{-1}A_{\mathcal{L}} +\frac{1}{\alpha}A_{\mathcal{G}}^\top A_{\mathcal{G}} -\frac{\beta^2}{\alpha}U_A U_A^\top -\frac{2}{\alpha} A_{\mathcal{L}}^\top D_{\mathcal{L}}^{-2}A_{\mathcal{L}} \\&-\frac{1}{2\alpha^2 u} A_{\mathcal{G}}^\top A_{\mathcal{G}}- A_{\mathcal{L}}^\top D_{\mathcal{L}}^{-1}(\frac{u}{2}I + D_{\mathcal{L}}^{-1})^{-1}D_{\mathcal{L}}^{-1} A_{\mathcal{L}}\\
     \succeq &  
     A_{\mathcal{L}}^\top  \left[ (1-\frac{D_{\max}}{2\alpha} - \frac{2 }{\alpha D_{\min} })I -(\frac{u}{2} D_{\mathcal{L}} + I)^{-1}   \right]D_{\mathcal{L}}^{-1} A_{\mathcal{L}} 
     \\
    & +\frac{1}{2\alpha} \left(   A_{\mathcal{L}}^\top A_{\mathcal{L}}  + A_{\mathcal{G}}^\top A_{\mathcal{G}}-2\beta^2 U_A U_A^\top \right)\\
    \succeq  &   \left[ (1-\frac{D_{\max}}{2\alpha} - \frac{2 }{\alpha D_{\min} }) -\frac{1}{1+ u/2 \cdot D_{\min}}  \right] A_{\mathcal{L}}^\top D_{\mathcal{L}}^{-1} A_{\mathcal{L}} 
     \\
     & + \frac{1}{2\alpha}  U_A\left( \Sigma_A \Sigma_A- 2\beta^2     I      \right)U_A^\top \succeq 0
    \end{align*}
   where $D_{\min} = \min_{i\in \mathcal{L}} D_i$ and $D_{\max} = \max_{i\in \mathcal{L}} D_i$. Thus $H_2 \succeq 0$.

By the arguments 1) - 4) above and equation (\ref{drd}),  we have 
$\be^\top R(\bd) \, \be \geq 0$ for any $\be$, which shows that $R(\bd)\succeq 0$. Thus Lemma \ref{le-exp} is proved.

}

\section{Proof of Proposition \ref{pro:bound}} \label{app:pro}
    
Consider the following Lagrangian function of the reduced OLC problem (\ref{op2}) with the dual variables $\bpi$ and $\bnu$:
\begin{align} \label{eq:lagred}
\begin{split}
     L(\bd,\bw,\bP,\bpsi,\bpi,\bnu)  &= c(\bd) + \frac{1}{2}\bw^\top D \bw \\
    & + \bpi^\top (\bd -\bP^{in} + D\bw + A\bP)\\
    &+ \bnu^\top (\bd-\bP^{in}+ S\bpsi)
\end{split}
\end{align}

To deal with the  non-uniqueness of the saddle points, define $ \tilde{\bP}: = U_A^\top \bP $ and $\tilde{\bpsi}: = U_S^\top \bpsi$ based on the compact singular value decomposition (\ref{svd}) (\ref{egd}), thus we have 
\begin{align*}
    A\bP =  V_A\Sigma_A  \tilde{\bP},\quad S \bpsi = U_S \Sigma_S \tilde{\bpsi} 
\end{align*}
and the corresponding optimal $\tilde{\bP}_*$ and $\tilde{\bpsi}_*$ are unique. Substitute $V_A\Sigma_A  \tilde{\bP}$ and $ U_S \Sigma_S \tilde{\bpsi}$ for $A\bP$ and $S\bpsi$ in the Lagrangian function (\ref{eq:lagred}), respectively.
Then the KKT conditions of the reduced OLC problem (\ref{op2}) are given by
\begin{subequations}\label{eq:KKT}
\begin{align} 
 \frac{\partial L}{\partial \bd} & = \nabla c(\bd_*) + \bpi_* + \bnu_* =0 \\
  \frac{\partial L}{\partial \bw}& = D\bw_* + D \bpi_* =0 \\
    \frac{\partial L}{\partial \tilde\bP} & =  \Sigma_AV_A^\top \bpi_*= 0 \\
 \frac{\partial L}{\partial \tilde\bpsi} & = \Sigma_S U_S^\top\bnu_*=0 \\
  \frac{\partial L}{\partial \bpi} & = \bd_*+D\bw_*+V_A\Sigma_A  \tilde{\bP}_*-\bP^{in} = 0 \\
 \frac{\partial L}{\partial \bnu}& = \bd_*+ U_S \Sigma_S \tilde{\bpsi}_* - \bP^{in}=0. 
\end{align}
\end{subequations}

Define $\by_*:= [\bd_*; \bw_*; \tilde\bP_*; \tilde \bpsi_*; \bpi_*; \bnu_*]$ as the optimal solution satisfying the KKT conditions (\ref{eq:KKT}).
Then the KKT conditions (\ref{eq:KKT}) can be equivalently rewritten as the compact form:
\begin{align*}
    \bm{h}(\by_*(t)) = E\bP^{in}(t) 
\end{align*}
where $E: = [\bm{0}, \bm{0}, \bm{0}, \bm{0}, I, I]^\top $. Thus we have 
\begin{align*}
  \bm{\nabla}_{\by_*}\!  \bm{h} \cdot \frac{d{\by}_*(t)}{dt} = E\frac{d{\bP}^{in}(t)}{dt}
\end{align*}
where 
\begin{align*}
    \bm{\nabla}_{\by_*}\!  \bm{h} = \begin{bmatrix}
    \nabla^2 c(\bd_*) & 0 & 0 & 0 & I & I\\
    0 & D & 0 & 0 & D & 0\\
    0 & 0 & 0& 0 &  \Sigma_S V_A^\top & 0\\
    0 & 0 & 0 & 0 & 0 & \Sigma_S U_S^\top \\
    I & D & V_A \Sigma_A & 0 & 0 & 0\\
    I & 0 & 0 & U_S \Sigma_S & 0 & 0 
    \end{bmatrix}.
\end{align*}
Since $u I \preceq \nabla^2 c(\bd_*)\preceq \ell I$ for any $\bd_*$ due to Assumption \ref{ass:strong}, matrix $  \bm{\nabla}_{\by_*}\!  \bm{h} $ is nonsingular according to \cite[Theorem 3.2]{ref:nonsing}. Thus we have
\begin{align*}
  ||\frac{d{\by}_*(t)}{dt}||\leq \frac{||E||}{\sigma_{\min}(  \bm{\nabla}_{\by_*}\!  \bm{h}  )} \cdot ||\frac{d{\bP}^{in}(t)}{dt}|| \leq \frac{\sqrt{2}}{\gamma_h}\cdot||\frac{d{\bP}^{in}(t)}{dt}|| 
\end{align*}
where $\sigma_{\min}(  \bm{\nabla}_{\by_*}\!  \bm{h}  )$ is the smallest singular value of $\bm{\nabla}_{\by_*}\!  \bm{h} $, and the second inequality  is due to  \cite[Proposition 2.2]{ref:smsing}, which shows that $\sigma_{\min}(  \bm{\nabla}_{\by_*}\!  \bm{h}  )\geq \gamma_h$ for some positive constant $\gamma_h$ that depends on $D, u, \ell, U_A\Sigma_A, U_S\Sigma_S$.

The rest of the proof is to bound $ ||\frac{d{\bz}^*(t)}{dt}||_Q$ by $ ||\frac{d{\by}_*(t)}{dt}||$. Since the primal variables of the  equilibrium points are optimal solutions to the reduced OLC problem (\ref{op2}), we have the following relation  between $\bz^* :=\left[\bd^*; \bP^*;\bpsi^*;\bw_{\mathcal{G}}^*; \bmu^* \right]$ and $\by_*:=[\bd_*; \bw_*; \tilde\bP_*; \tilde \bpsi_*; \bpi_*; \bnu_*]$: 
\begin{align*}
  \bd^* &=\bd_*,\ U_A^\top \bP^* = \tilde\bP_*,\ U_S^\top \bpsi^* = \tilde \bpsi_*,\\ \bw_{\mathcal{G}}^*& = [I,\, \bm{0}]\, \bw_*, \ \bmu^* = \nabla c(\bd_*) - \bw_*  
\end{align*}
where the last equality  is due to equation (\ref{gra:d2}). Therefore, for the time derivatives, we have
\begin{align*}
\dot{\tilde{\bz}}^* : &= \begin{bmatrix}
\dot{\bd}^*\\
U_A^\top \dot\bP^*\\
U_S^\top \dot\bpsi^*\\
\dot\bw_{\mathcal{G}}^*\\
 \dot{\bmu}^*
\end{bmatrix} \\
&= \begin{bmatrix}
\dot{\bd}_*\\ \dot{\tilde\bP}_*\\
\dot{\tilde \bpsi}_*\\
[I,\, \bm{0}]\, \dot{\bw}_*\\
\nabla^2 c(\bd_*) \dot{\bd}_* - \dot{\bw}_*
\end{bmatrix} = \underbrace{\begin{bmatrix}
 E_d \\E_P\\E_\psi\\ [I,\, \bm{0}]E_{w}\\ \nabla^2 c(\bd_*) E_d -E_w
\end{bmatrix}}_{:= E_y(\bd_*)}\dot{\by}_*
\end{align*}
where $E_d, E_P, E_\psi,E_{w} $ are corresponding constant matrices, e.g., $E_d: = [I, \bm{0}, \bm{0}, \bm{0}, \bm{0}, \bm{0} ]$.

For any $\dot{\bz}^*$, we have 
\begin{align*}
   & ||\dot{\bz}^*||_Q^2 =  \dot{\bz}^{*\top} Q \dot{\bz}^* \\
   =& \dot{\tilde{\bz}}^{*\top}\underbrace{ \begin{bmatrix}
    \alpha I & \bo & \bo & \bo & I\\
    \bo & \alpha I & \bo & U_A^\top A_\mathcal{G}^\top  & \bo\\
    \bo & \bo & \alpha I & \bo & -\beta \Sigma_SU_S^\top\\
    \bo & A_\mathcal{G}U_A & \bo & \alpha I &  \bo\\
    I &  \bo & -\beta U_S\Sigma_S & \bo& \alpha I
    \end{bmatrix}}_{:=\hat{Q}} \dot{\tilde{\bz}}^*\\
    = & (E_y(\bd_*) \dot{\by}_*)^\top \hat{Q}   E_y(\bd_*) \dot{\by}_* = \dot{\by}_*^\top \left(E_y(\bd_*)^\top \hat{Q} E_y(\bd_*)  \right)\dot{\by}_*.
\end{align*}
Since $u I \preceq \nabla^2 c(\bd_*) \preceq \ell I$ for any $\bd_*$ due to Assumption \ref{ass:strong}, it can show that matrix $E_y(\bd_*)^\top \hat{Q} E_y(\bd_*)$ is positive semi-definite and its norm is upper bounded for any $\bd_*$, i.e.,  there exists a positive constant $\gamma_y$ such that $$||E_y(\bd_*)^\top \hat{Q} E_y(\bd_*) ||\leq \gamma_y.$$

As a result, we obtain
\begin{align*}
  ||\frac{d{\bz}^*}{dt}||_Q \leq \sqrt{\gamma_y} \cdot ||\frac{d{\by}_*}{dt}|| \leq   \frac{\sqrt{2\gamma_y}}{\gamma_h}\cdot||\frac{d{\bP}^{in}(t)}{dt}||.
\end{align*}
Let $\eta: = \frac{\sqrt{2\gamma_y}}{\gamma_h}$, then Proposition \ref{pro:bound} is proved.

\section{Theorem on Inaccurate Damping} \label{pf-robustD}

\begin{theorem}\label{thm-robustD}
Under Assumption \ref{assumption:feasibility} and \ref{ass:strong}, and the following conditions are met: 

\noindent i) Infinitely large step sizes $\epsilon_{d_i}$ are used for \eqref{cm:d}, which is then reduced to the following algebraic equation:
\begin{eqnarray}\nonumber
-c_i^\prime\left(d_i\right)+\eta_i\omega_i+\frac{\epsilon_{\mu_i}}{K_i}r_i-\gamma_i^++\gamma_i^-  =0.
\end{eqnarray}

\noindent ii) An inaccurate $\tilde D_i = D_i + \delta a_i$ is used instead of $D_i$ in \eqref{cm:r}, and the inaccuracy $\delta a_i$ satisfies:
\begin{eqnarray}\label{cond:uncertainD}
\delta a_i \in 2\left( \underline d' - \sqrt{\underline d'^2 + \underline d' D_{\mathrm{min}}}, ~\underline d' + \sqrt{\underline d'^2 + \underline d' D_{\mathrm{min}}}\right)
\end{eqnarray}
where $\underline d' := 1/\ell$ and $D_{\mathrm{min}}:=\min_{i \in \mathcal{N} } D_i$.  

\noindent \textcolor{black}{iii) Every node $i \in \mathcal{N}$ has adequate load control capacity such that its control action $d_i(t)$ never hits the limit of $[\underline d_i,~\overline d_i]$ at any time $t$.} 

Then the closed-loop system \eqref{dynamic} and \eqref{cm} globally asymptotically converges to a point $\left(\bd^*, \bw^*, \bP^*,  \bpsi^*, \bga^*, \br^*,\bsi^*\right)$, where $\left(\bd^*, \bw^*, \bP^*,  \bpsi^*\right)$ is an optimal solution of problem \eqref{op}.
\end{theorem}

\begin{proof}

For $i\in\mathcal{N}$, an inaccurate damping coefficient $\tilde D_i = D_i + \delta a_i$ is used instead of $D_i$ in \eqref{cm:r}. The closed-loop system \eqref{dynamic}, \eqref{cm} is then equivalent to (\ref{sol:lam}), \eqref{gra} except that \eqref{gra:mu} becomes
\begin{align}
\begin{split}
    \dot{\mu}_i= \epsilon_{\mu_i}\left(P^{in}_i-d_i + \delta a_i \omega_i -\sum_{j:ij\in\mathcal{E}_{in}}B_{ij}\left(\psi_{i}-\psi_j\right)  \right.\\
		\phantom{=\;\;}\left.  +\sum_{k:ki\in \mathcal{E}_{in}}B_{ki}\left(\psi_k-\psi_i\right)\right) \label{eq:virtual_freq_uncertainD}
\end{split}
\end{align} 
with the additional term $\delta a_i \omega_i$.

\add{By condition iii) of Theorem \ref{thm-robustD}}, we have $d_i(t) \in (\underline d_i,\overline d_i)$ and $\gamma_i^+(t) \equiv \gamma_i^-(t) \equiv 0$, for all $t\geq 0$, given that their initial values satisfy $d_i(0) \in (\underline d_i,\overline d_i)$ and $\gamma_i^+(0) = \gamma_i^-(0)=0$. Thus the dynamics of $\bga^+$ and $\bga^-$ can be ignored from \eqref{gra}. 
Further by condition i), the control law \eqref{gra:d} is modified as (\ref{md-l})
\begin{eqnarray} \label{md-l}
d_i &=& \left(c_i^\prime\right)^{-1} \left(\omega_i+\mu_i \right) \qquad \forall i \in \mathcal{N} \label{eq:d_control_instant}
\end{eqnarray}

Define $\bze:=[\bP; \bpsi; \bw_{\mathcal{G}};\bmu; \bsi]$ and $
\tilde L(\bze) := \min_{\bd} \overline L(\bd, \bze) $, 
where the minimizer $\bd$ is given by \eqref{eq:d_control_instant} and  $\frac{\partial\overline{L}}{\partial \bd}(\bd,\bze) = 0 $.

The modified closed-loop system with inaccurate $D_i$, described by \eqref{sol:lam}, \eqref{gra:lambda}, \eqref{gra:P}, \eqref{gra:psi}, \eqref{gra:sigma_plus}, \eqref{gra:sigma_minus}, \eqref{eq:virtual_freq_uncertainD}, \eqref{eq:d_control_instant}, can be written as:   
\begin{align} \label{dyn_uncertianD}
\begin{split}
&\dot{\bP}=-\Xi_P\frac{\partial\tilde{L}}{\partial \bP}, \quad \dot{\bpsi}=-\Xi_\psi\frac{\partial\tilde{L}}{\partial \bpsi}, \quad
\dot{\bw}_\mathcal{G}=\Xi_{\omega_\mathcal{G}}\frac{\partial\tilde{L}}{\partial \bw_\mathcal{G}}\\ 
&\dot{\bmu}=\Xi_\mu\left[\frac{\partial\tilde{L}}{\partial \bmu} + \delta A \bw \right], \quad \ \,
\dot{\bsi}=\Xi_{\sigma}\left[\frac{\partial\tilde{L}}{\partial {\bsi}}\right]^+_{\bsi}
\end{split}
\end{align}
where $\delta A:=\mathrm{diag}(\delta a_i)_{i\in \mathcal{N}}$. The system \eqref{dyn_uncertianD} can be written more compactly as:
\begin{eqnarray} \label{dyn_uncertianD_compact}
\dot \bze &=& \Xi_\zeta \left[ f(\bze) \right]_{\bsi}^+
\end{eqnarray}
where $\Xi_\zeta:=\mathrm{blockdiag}(\Xi_P,\Xi_\psi, \Xi_{\omega_\mathcal{G}}, \Xi_\mu, \Xi_\sigma)$ and 
\begin{eqnarray}\nonumber
f(\bze):=\left[ -\frac{\partial\tilde{L}}{\partial \bP}^\top,  -\frac{\partial\tilde{L}}{\partial \bpsi}^\top,  \frac{\partial\tilde{L}}{\partial \bw_\mathcal{G}}^\top, \left(\frac{\partial\tilde{L}}{\partial \bmu} + \delta A \bw \right)^\top,  \frac{\partial\tilde{L}}{\partial {\bsi}}^\top\right]^\top.
\end{eqnarray}
Note that in the vector $\bw=[\bw_\mathcal{G}; \bw_\mathcal{L}]$, only $\bw_\mathcal{G}$ is a variable of the system \eqref{dyn_uncertianD} or \eqref{dyn_uncertianD_compact}, whereas $\bw_\mathcal{L}$ is the abbreviation of a vector-valued function $\bw_\mathcal{L}(\bze)$ defined by the equation:
\begin{eqnarray}\nonumber
P^{in}_i \!-\! d_i(\omega_i \!+\! \mu_i) \!-\! D_i \omega_i \!-\! \sum_{j:ij\in \mathcal{E}}P_{ij}\!+\! \sum_{k:ki\in \mathcal{E}}P_{ki} = 0,~ \forall i \in \mathcal{L}
\end{eqnarray}
where $d_i (\omega_i + \mu_i)$ is defined by \eqref{eq:d_control_instant}.

The rest of the proof follows the same technique as the proof of \cite[Theorem 15]{the}, and therefore we only provide a sketch for it. 
Consider a Lyapunov function candidate:
\begin{eqnarray}\nonumber
U(\bze) = \frac{1}{2} \left(\bze - \bze^*\right)^\top \Xi_{\zeta}^{-1} \left(\bze- \bze^*\right).
\end{eqnarray}
We first show that under the dynamics \eqref{dyn_uncertianD_compact}, the time derivative of $U(\bze)$ is upper-bounded by:
\begin{eqnarray}\nonumber
\dot U(\bze) \leq \int_0^1 (\bze-\bze^*)^\top [H(\bze(s))](\bze-\bze^*) ds
\end{eqnarray} 
where $\bze(s) = \bze^* +s (\bze- \bze^*)$, and $H(\bze)$ is a matrix which is zero everywhere except:

\noindent i) a block submatrix corresponding to variables $(\bP, \bmu_\mathcal{L})$, which is the same as $H_{P,\omega_\mathcal{L}}(z)$ in \cite{the}; 

\noindent ii) a block submatrix corresponding to variables $(\bmu_\mathcal{G}, \bw_\mathcal{G})$, which is the same as $H_{\omega_\mathcal{G}, \nu_\mathcal{G}}(z)$ in \cite{the}.

It is shown in \cite{the} that under condition \eqref{cond:uncertainD}, the matrix $H(\bze)$ is negative semi-definite. Applying the invariance principle, the convergence result in Theorem \ref{thm-robustD} can be proved. 

\end{proof}

\textbf{Discussions.}  1)\emph{Why are $\bga^+$ and $\bga^-$ ignored?} 
If \eqref{gra:gamma_plus}--\eqref{gra:gamma_minus} are considered, then instead of proving negative semi-definiteness of the block submatrices corresponding to $(\bP, \bmu_\mathcal{L})$ and $(\bmu_\mathcal{G}, \bw_\mathcal{G})$, we have to prove negative semi-definiteness of the block submatrices corresponding to $(\bP, \bmu_\mathcal{L}, \bga^+_\mathcal{L}, \bga^-_\mathcal{L})$ and $(\bmu_\mathcal{G}, \bw_\mathcal{G}, \bga^+_\mathcal{G}, \bga^-_\mathcal{G})$. However, one can show that the latter two larger block submatrices have strictly positive eigenvalues for arbitrarily small $\delta A$, which makes the proof technique fail.

 2) \emph{Why is the control law modified from the derivative form \eqref{gra:d} to the stationary form \eqref{eq:d_control_instant}?}
With the derivative form \eqref{gra:d}, one can show that in $H(\bze)$, the block at the diagonal position corresponding to $\bmu$ is zero, and hence it is impossible to make $H(\bze)$ negative semi-definite when the off-diagonal blocks containing $\delta A$ are non-zero.



\begin{thebibliography}{1}
 \bibitem{3m}
Y. G. Rebours, D. S. Kirschen, M. Trotignon and S. Rossignol, ``A survey of frequency and voltage control ancillary services: Part I: Technical features," \emph{IEEE Trans. Power Syst.}, vol. 22, no. 1, pp. 350-357, Feb. 2007.

\bibitem{r_r}
B. M. Sanandaji, T. L. Vincent, and K. Poolla, ``Ramping rate flexibility
of residential HVAC loads," \emph{IEEE Trans. Sustain. Energy}, vol. 7, no. 2,
pp. 865-874, Apr. 2016.

 \bibitem{ev1}
H. Liu, Z. Hu, Y. Song, and J. Lin, ``Decentralized vehicle-to-grid control
for primary frequency regulation considering charging demands,"
\emph{IEEE Trans. Power Syst.}, vol. 28, no. 3, pp. 3480-3489, Aug. 2013.

\bibitem{ev2}
H. Liu, J. Qi, J. Wang, P. Li, C. Li and H. Wei, ``EV dispatch control for supplementary frequency regulation considering the expectation of EV owners," \emph{IEEE Trans. Smart Grid}, vol. PP, no. 99, pp. 1-1.

\bibitem{hvac1}
 I. Beil, I. Hiskens, and S. Backhaus, ``Frequency regulation from commercial
building HVAC demand response," \emph{Proceedings of the IEEE}, vol. 104, no. 4, pp. 745-757, Apr. 2016.


\bibitem{bat1}
A. Oudalov, D. Chartouni, and C. Ohler, ``Optimizing a battery energy storage system for primary frequency control," \emph{IEEE Trans. Power
Syst.}, vol. 22, no. 3, pp. 1259-1266, Aug. 2007. 

\bibitem{bat2}
U. Akram and M. Khalid, ``A coordinated frequency regulation framework based on hybrid battery-ultracapacitor energy storage technologies," \emph{IEEE Access}, vol. PP, no. 99, pp. 1-1.

\bibitem{tcl}
 Y. J. Kim, L. K. Norford, and J. L. Kirtley, ``Modeling and analysis
of a variable speed heat pump for frequency regulation through direct
load control," \emph{IEEE Trans. Power Syst.}, vol. 30, no. 1, pp. 397-408,
Jan. 2015.

\bibitem{p1}
Y. Lin, P. Barooah, S. Meyn, and T. Middelkoop, ``Experimental evaluation
of frequency regulation from commercial building HVAC systems," \emph{IEEE
	Trans. Smart Grid}, vol. 6, no. 2, pp. 776-783, Mar. 2015.

\bibitem{p2}
G. Heffner, C. Goldman, and M. Kintner-Meyer, ``Loads providing ancillary services: review of international experience," Lawrence Berkeley National Laboratory, Berkeley, CA, USA, \emph{Tech. Rep.}, 2007.

\bibitem{p3}
D. Hammerstrom et al., ``Pacific Northwest GridWise testbed demonstration projects, part II: Grid friendly appliance project," Pacific Northwest
Nat. Lab., Richland, WA, USA, \emph{Tech. Rep. PNLL-17079}, Oct. 2007.

\bibitem{cen1}
M. Aldeen and H. Trinh, ``Load frequency control of interconnected power systems
via constrained feedback control scheme," \emph{Computers Elect. Engng}, vol. 20, no. 1, pp. 71-88, 1994.

\bibitem{cen2}
H. Shayeghi, H. Shayanfar, ``Application of ANN technique based on $\mu$-synthesis to load frequency control of interconnected power system," \emph{Electr
Power Energy Syst}, vol. 28, no. 7, pp. 503-511, Sept. 2006.

\bibitem{lr2}
A. Pappachen and A. Peer Fathima, ``Critical research areas on load
frequency control issues in a deregulated power system: A state-of-the-art-of-review," \emph{Renewable Sustain. Energy Reviews}, vol. 72, pp. 163-177, May 2017.

\bibitem{dec1}
C. Wu and T. Chang, ``ADMM approach to asynchronous distributed frequency-based load control," in \emph{Proc. 2016 IEEE Global Conference on Signal and Information Processing (GlobalSIP)}, 2016, pp. 931-935.

\bibitem{dec2}
S. Abhinav, I. Schizas, F. Ferrese and A. Davoudi, ``Optimization-based AC microgrid synchronization," \emph{IEEE Trans. Ind. Informat.}, vol. 13, no. 5, pp. 2339-2349, Oct. 2017.

\bibitem{dec3}
X. Zhang and A. Papachristodoulou, ``A real-time control framework for smart power networks: Design methodology and stability," \emph{Automatica}, vol. 58, pp. 43-50, Aug. 2015.

\bibitem{pi1}
M. Andreasson, D. V. Dimarogonas, H. Sandberg, and K. H. Johansson, ``Distributed control of networked dynamical systems: Static feedback, integral action and consensus," \emph{IEEE Trans. Autom. Control}, vol. 59,
no. 7, pp. 1750-1764, Jul. 2014.

\bibitem{the}
E. Mallada, C. Zhao and S. Low, ``Optimal load-side control for frequency regulation in smart grids," \emph{IEEE Trans. Autom. Control}, vol. 62, no. 12, pp. 6294-6309, Dec. 2017.

\bibitem{re1}
N. Li, C. Zhao, and L. Chen, ``Connecting automatic generation control
and economic dispatch from an optimization view," \emph{IEEE Trans. Control
Netw. Syst.}, vol. 3, no. 3, pp. 254-264, Sept. 2016.


\bibitem{sm}
C. Zhao, U. Topcu, N. Li, and S. Low, ``Design and stability of load-side primary frequency control in power systems," \emph{IEEE Trans. Autom. Control}, vol. 59, no. 5, pp. 1177-1189, May 2014.




\bibitem{dc}
B. Zhang and Y. Zheng, \emph{Advanced Electric Power Network Analysis}, 1st ed. Cengage Learning Asia, Nov. 2010.

\bibitem{dc2}
B. Stott, J. Jardim, and O. Alsac, ``DC power flow revisited," \emph{IEEE Trans. Power Syst.}, vol. 24, no. 3, pp. 1290-1300, Aug. 2009.

\bibitem{pp}
D. Feijer and F. Paganini, ``Stability of primal-dual gradient dynamics and applications to network optimization", \emph{Automatica}, vol. 46, no. 12, pp. 1974-1981, Dec. 2010.


\bibitem{boyd2004convex}
S.~P. Boyd, L.~Vandenberghe, Convex Optimization, Cambridge University Press, 2004.




\bibitem{pf1}
A. Cherukuri, E. Mallada, and J. Cortés, ``Asymptotic convergence of constrained primal-dual dynamics," \emph{Systems \& Control Letters}, vol. 87, pp. 10-15, Jan. 2016.








\bibitem{Hao}
K. Zhang, W. Shi, H. Zhu, E. Dall'Anese and T. Başar, ``Dynamic power distribution system management with a locally connected communication network," \emph{IEEE Journal of Selected Topics in Signal Processing}, vol. 12, no. 4, pp. 673-687, Aug. 2018.

\bibitem{steven}
Y. Tang, K. Dvijotham and S. Low, ``Real-time optimal power flow,"  \emph{IEEE Trans. on Smart Grid}, vol. 8, no. 6, pp. 2963-2973, Nov. 2017.

\bibitem{guan}
G. Qu and N. Li, ``On the exponential stability of primal-dual gradient dynamics," \emph{arXiv preprint}, arXiv:1803.01825, 2018.

\bibitem{expp1}
S. K. Niederl\" ander, F. Allg\" ower and J. Cort\' es, "Exponentially fast distributed coordination for nonsmooth convex optimization," \emph{2016 IEEE 55th Conference on Decision and Control (CDC)}, Las Vegas, NV,  pp. 1036-1041, 2016.

\bibitem{expp2}
J. Cort\' es and S. K. Niederl\"ander, ``Distributed coordination for nonsmooth convex optimization via saddle-point dynamics," \emph{Journal of Nonlinear Science},
pp. 1-26, 2018.

\bibitem{pst1}
K. W. Cheung, J. Chow, and G. Rogers, Power System Toolbox ver. 3.0., Rensselaer Polytechnic Institute and Cherry Tree Scientific Software, 2009.

\bibitem{pst2}
Power System Toolbox Webpage [Online]. Available: \url{http://www.eps.ee.kth.se/personal/vanfretti/pst/Power_System_Toolbox_Webpage/PST.html}.


\bibitem{rea}
P. J. Douglass, R. Garcia-Valle, P. Nyeng, J. Østergaard, and M. Togeby,
``Smart demand for frequency regulation: Experimental results,"
\emph{IEEE Trans. Smart Grid}, vol. 4, no. 3, pp. 1713-1720, Sept. 2013.

\bibitem{numeri}
M. Benzi, G. H. Golub, and J. Liesen.  Numerical solution of saddle point problems. \emph{Acta Numerica}, pp. 1-137, 2005.

\bibitem{spectr}
N. I. M. Gould,  V. Simoncini, ``Spectral analysis of saddle point matrices with indefinite leading blocks," \emph{SIAM Journal on Matrix Analysis and Applications}, vol. 31, no. 3, pp. 1152-1171, 2010.

\bibitem{ref:nonsing}
M. Benzi, G. H. Golub, and J. Liesen.  Numerical solution of saddle point problems. \emph{Acta Numerica}, pp. 1-137, 2005.

\bibitem{ref:smsing}
 N. I. M. Gould,  V. Simoncini, ``Spectral analysis of saddle point matrices with indefinite leading blocks," \emph{SIAM Journal on Matrix Analysis and Applications}, vol. 31, no. 3, pp. 1152-1171, 2010.

\end{thebibliography}
\end{document}